\documentclass[12pt]{amsart}

\usepackage{amsmath} 
\usepackage{amssymb} 
\usepackage{amsthm} 
\usepackage{dsfont} 
\usepackage{graphicx} 
\usepackage{enumerate} 
\usepackage{url} 

\newtheorem{theorem}{Theorem}[section]

\newtheorem{definition}[theorem]{Definition}
\newtheorem{example}[theorem]{Example}
\newtheorem{lemma}[theorem]{Lemma}

\begin{document}

\title[Generalized Marriage Theorem]{A Generalized Marriage Theorem}

\author{Thomas Fischer}
\address{Frankfurt am Main, Germany.}

\email{dr.thomas.fischer@gmx.de}

\date{}     


\begin{abstract}
We consider a set-valued mapping on a simple graph and ask for the existence of a disparate selection. 
The term disparate is defined in the paper and we present a sufficient and necessary condition for the 
existence of a disparate selection. This approach generalizes the classical marriage theorem of Hall. 
We define the disparate kernel of the set-valued mapping and provide calculation methods for the 
disparate kernel and a disparate selection. Our main theorem is applied to a result of Ryser on the 
completion of partially prepopulated Latin squares and we derive Hall's marriage theorem.
\end{abstract}




\maketitle

                                        %
                                        %
\section{Introduction} \label{S:intro}
We consider a set-valued mapping defined on a simple graph into a finite set. We ask for the existence 
of a selection with a specific property which we call disparate. This model generalizes a classical result 
of P. Hall \cite{HalP} who considered the case when the domain is a complete graph (compare 
van Lint and Wilson \cite{vLW} for a definition of complete graph). Our model includes list coloring 
problems introduced among others by Erd{\"o}s et al. \cite{ERT}. Survey articles on recent results and 
applications of graph and list coloring problems are due to Formanowicz and Tana\'{s} \cite{FT} and 
Sankar et al. \cite{SFRN}. 

The main result of this paper is a sufficient and necessary condition for the existence of a disparate 
selection. This condition is described by the existence of a distributed Hall collection. A Hall collection 
is a condition which had been investigated by several authors using a different terminology and is 
known to be necessary for the existence of a disparate selection (see e.g. Hilton and Johnson \cite{HJ2},
Cropper \cite{Crop} or Bobga et al. \cite{Bob2}). The existence of a distributed Hall collection is still 
necessary for the existence of a disparate selection, but also sufficient.

We introduce the terms disparate, complement mappings and cascades. The definitions 
and proofs, in particular the definition of the complement mapping, use ideas of 
Halmos and Vaughan \cite{HV} in their proof of the marriage theorem.

Based on this characterization we introduce the disparate kernel of a set-valued mapping. The 
disparate kernel describes the submapping of all disparate selections of a set-valued mapping.
We provide a method for the calculation of the disparate kernel and a method for the calculation of 
a disparate selection.

Also based on this characterization we introduce a condition on the simple graph, such that the 
existence of a Hall collection is equivalent to the existence of a disparate selection. From this result 
we derive a result on prepopulated Latin squares of Ryser \cite{Rys} and the before mentioned result 
of P. Hall \cite{HalP}.

Finally in this section we collect some basic terms and notations. The symbol $\sharp$ denotes the 
number of elements (cardinality) of a finite set. The expression $X \times Y$ denotes the cartesian 
product of $X$ and $Y$. The set of positive integers is denoted by $\mathbb{N}$. 

For the concept of set-valued mappings compare Berge \cite{Ber}. We do not require $F$ to have 
nonempty image sets as a general condition. But the assumptions in our lemmas and theorems 
will imply that $F$ has nonempty image sets.

For a set-valued mapping $F: X \longrightarrow 2^Y$ and a given set $W \subset X$ we define 
$F(W) = \cup_{x \in W} F(x)$. The restriction of $F$ on a subset $W \subset X$ is denoted by 
$F_{\mid W}$. The graph of $F$ is defined by $G(F) = \{(x,y) \in X \times Y \mid y \in F(x) \}$. 
A submapping of $F$ is another set-valued mapping $F^\prime: X \longrightarrow 2^Y$ with the 
property $F^\prime (x) \subset F(x)$ for each $x \in X$ which is written $G(F^\prime) \subset G(F)$. 

A selection for $F$ is a point-valued mapping $s: X \longrightarrow Y$ such that $s(x) \in F(x)$ for 
each $x \in X$. Given a submapping $F^\prime$ and a selection $s$ of $F$ such that 
$F^\prime(x) = \{s(x)\}$ we identify $s$ with the mapping $F^\prime$. In this sense $F^\prime$ 
is a point-valued mapping.

                                        %
                                        %
\section{Disparate Selections} \label{S:disparate}
Let $(X,E)$ be an undirected, simple graph with finite vertex set $X$ and edge set $E$. Let $Y$ be a 
finite set and let $F: X \longrightarrow 2^Y$ be a set-valued mapping.

\begin{definition} \label{D:disparate}
The points $(x,y), (x^\prime, y^\prime) \in X \times Y$ are called disparate if
$(x,y) \ne (x^\prime, y^\prime)$ and $y=y^\prime$ implies $\{x,x^\prime\} \notin E$.
\end{definition}

Two points $(x,y), (x^\prime, y^\prime) \in X \times Y$ are not disparate if and only if they are
identical or $y=y^\prime$ and $\{x,x^\prime\} \in E$. We extend the definition of disparate points 
to disparate sets.

\begin{definition} \label{D:disparateSet}
A set $A \subset X \times Y$ is called disparate if any two points $a, a^\prime \in A$, $a \ne a^\prime$, 
are disparate.
\end{definition}

We use the concept of disparate sets for the definition of disparate selections.

\begin{definition} \label{D:disparateSel}
A point-valued mapping $s: X \longrightarrow Y$ is called to be disparate if the graph $G(s)$ of $s$
is a disparate set.
\end{definition}

In this paper we ask for the existence of a disparate selection for the set-valued mapping $F$. The main 
result describes a condition which is sufficient and necessary for the existence of a disparate selection 
of $F$.

In the context of coloring problems (when $Y$ denotes the set of colors) a disparate selection is often 
called a proper coloring (e.g. van Lint and Wilson \cite{vLW}). A disparate selection $s$ for $F$ 
represents a proper coloring of the graph $(X,E)$.

The term ``disparate selection" is a generalization of the term ``complete system of distinct 
representatives" of P. Hall \cite{HalP}. Other authors used the terms ``alldifferent" resp. ``transversal". 
The term ``alldifferent" is widely used in the theory of constraint satisfaction problems. See 
Dechter and Rossi \cite{DR} for an overview on constraint satisfaction problems and see 
Hoeve \cite{Hoe} for an overview on alldifferent constraints.

                                        %
                                        %
\section{The Complement Mapping} \label{S:complement}    
Based on disparate points we define the hull and the complement of a set in $X \times Y$. 
The hull and the complement are the basic prerequisites for the definition of the complement mapping.

\begin{definition} \label{D:hull}
The hull of a set $A \subset X \times Y$ is defined by 
\begin{align*}                                                   
hull(A) = \{ a \in X \times Y \mid 
& \mbox{ there exists $a^\prime \in A$  such that} \\
& \mbox{ $a$ and $a^\prime$ are not disparate } \}.
\end{align*}
\end{definition}

Elementary properties of the hull are $A \subset hull(A)$ and $A \subset B$ implies $hull(A) \subset hull(B)$
for any sets $A, B \subset X \times Y$.

\begin{definition} \label{D:complement}
The complement of a set $A \subset X \times Y$ is defined by 
\begin{align*}                                                   
compl(A) = \{ a \in X \times Y \mid 
& \mbox{ $a$ and $a^\prime$ are disparate } \\
& \mbox{ for each $a^\prime \in A$ } \}.
\end{align*}
\end{definition}

The hull and the complement of any set in $X \times Y$ describe a partition of $X \times Y$.

\begin{lemma} \label{L:32}
Let $A \subset X \times Y$. 
\begin{enumerate}[(i)] 
\item $hull(A) \cap compl(A) = \emptyset$.
\item $hull(A) \cup compl(A) = X \times Y$.
\end{enumerate}
\end{lemma}
\begin{proof}
The proof follows from Definitions \ref{D:hull} and \ref{D:complement}.
\end{proof}

Based on Lemma \ref{L:32} we state elementary properties of the complement. 
$A \cap compl(A) = \emptyset$ and $A \subset B$ implies $compl(B) \subset compl(A)$ for any sets 
$A, B \subset X \times Y$.

We obtain more properties of the complement.

\begin{lemma} \label{L:34}
Let $A \subset X \times Y$. $A \subset compl(compl(A))$.
\end{lemma}
\begin{proof}
By definition of the complement $a \in A$ implies $a, b$ are disparate for each $b \in compl(A)$. 
Again by definition of the complement $a \in compl(compl(A))$ if and only if $a, b$ are disparate 
for each $b \in compl(A)$. This shows the statement.
\end{proof}

\begin{lemma} \label{L:35}
Let $A, B \subset X \times Y$. The following statements are equivalent:
\begin{enumerate}[(i)] 
\item $A$ and $B$ are disparate sets and $B \subset compl(A)$.
\item $A \cap B = \emptyset$ and $A \cup B$ is disparate.
\end{enumerate}
\end{lemma}
\begin{proof}
``(i) $\Rightarrow$ (ii)"
$A \cap B = \emptyset$, since $A \cap B \subset A \cap compl(A)$. Let $v, w \in A \cup B$, $v \ne w$. 
W.l.o.g. we may assume $v \in A$ and $w \in B$. Using ``(i)", $w \in compl(A)$, i.e., $v$ and 
$w$ are disparate. \\
``(ii) $\Rightarrow$ (i)" 
Using ``(ii)", $A$ and $B$ are disparate sets. Let $b \in B$ and $a \in A$. Using ``(ii)", $a$ and $b$ 
are disparate, i.e., $b \in compl(A)$.
\end{proof}

We prove a decomposition lemma.

\begin{lemma} \label{L:decomposition1}
Let $W \subset X$, $V \subset X \backslash W$ and  $Z \subset X \times Y$.
\[
(W \cup V) \times Y 
= (W \times Y) \cup ((V \times Y) \cap hull(Z)) \cup ((V \times Y) \cap compl(Z))
\]
and $((W \times Y) \cup ((V \times Y) \cap hull(Z))) \cap ((V \times Y) \cap compl(Z)) = \emptyset$.
\end{lemma}
\begin{proof}
Both equations follow from $(W \cup V) \times Y = (W \times Y) \cup (V \times Y)$ and
Lemma \ref{L:32}.
\end{proof}

\begin{figure}
\includegraphics{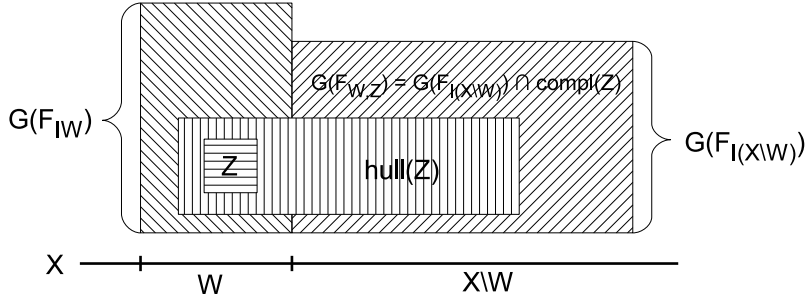}
\caption{The Complement Mapping} 
\label{Fig1}
\end{figure}

We introduce the complement mapping $F_{W,Z}: X \backslash W \longrightarrow 2^Y$ of $F$ 
depending on subsets $W \subset X$ and $Z \subset G(F_{\mid W})$. We define this mapping
based on a description of the graph
\[
G(F_{W,Z}) = G(F_{\mid (X \backslash W)}) \cap compl(Z).
\]
The definition of the complement mapping is illustrated in Fig. \ref{Fig1}. In particular, 
$F_{\emptyset , \emptyset} = F$. $F_{W,Z}$ is a submapping of $F_{\mid (X \backslash W)}$.  
In general, it can not be guaranteed that $F_{W,Z}$ has nonempty images. 

There are two main usages of the complement mapping in this paper. The first usage is where $W$ 
and $Z$ are singletons (compare Lemma \ref{L:complement1}). The second usage is illustrated in 
Lemma \ref{L:complement2} and Lemma \ref{L:H_W}.

\begin{lemma} \label{L:complement1}
Let $x \in X$ and $y \in F(x)$. The following statements are equivalent: 
\begin{enumerate}[(i)] 
\item $F$ admits a disparate selection $s$ such that $s(x)=y$.
\item $F_{\{x\},\{(x,y)\}}$ admits a disparate selection.
\end{enumerate}
\end{lemma}
\begin{proof} 
``(i) $\Rightarrow$ (ii)" Consider $s_{\mid (X \backslash \{x\})}$ and the definition of the 
complement mapping 
$G(F_{\{x\},\{(x,y)\}})=G(F_{\mid (X \backslash \{x\})}) \cap compl(\{(x,y)\})$.  \\
``(ii) $\Rightarrow$ (i)" Let $s^\prime$ be a disparate selection of $F_{\{x\},\{(x,y)\}}$. Define
a selection $s$ of $F$ by
\[
s(z) =
\begin{cases}
y, & \mbox{ if }z = x, \\
s^\prime (x), & \mbox{ if }z \ne x
\end{cases}
\]
for each $z \in X$. By definition of the complement mapping $s$ is a disparate selection of $F$ and 
$s(x) = y$.
\end{proof}

\begin{lemma} \label{L:complement2}
Let $W \subset X$ be a nonempty set and $Z \subset G(F_{\mid W})$. Let $s$ be a selection of $F$ 
such that $s_{\mid W}$ is disparate, $G(s_{\mid W}) \subset Z$, $s_{\mid (X \backslash W)}$ is 
disparate and $s_{\mid (X \backslash W)}$ is a selection of $F_{W,Z}$. $s$ is a disparate selection 
of $F$.
\end{lemma}
\begin{proof}
Apply Lemma \ref{L:35} with $A=G(s_{\mid W})$ and $B=G(s_{\mid (X \backslash W)})$.
\end{proof}

                                        %
                                        %
\section{Hall Collections} \label{S:collection}
We introduce Hall collections of $F$ and split the set $X$ into subsets $W$ and $X \backslash W$. 
We define submappings of $F_{\mid W}$ and $F_{\mid (X \backslash W)}$ and derive Hall 
collections of these submappings.

Let $A_V$ be a subset of $G(F_{\mid V})$ for each nonempty set $V \subset X$. We call 
\[
\{ A_V \subset G(F_{\mid V}) \mid V \subset X, V \ne \emptyset \}
\]
a collection of $F$, i.e., in a collection of $F$ we associate to each nonempty set $V \subset X$ a 
subset $A_V$ of $G(F_{\mid V})$.
              
Let $\mathcal{H} = \{ A_V \subset G(F_{\mid V}) \mid V \subset X, V \ne \emptyset \}$ be a
collection of $F$ and $W \subset X$. We derive two other collections from $\mathcal{H}$. The 
restriction of $\mathcal{H}$ on $W$ is defined by
\[
\mathcal{H}_{\mid W} = \{ A_V \in \mathcal{H} \mid V \subset W, V \ne \emptyset \}.
\] 
$\mathcal{H}_{\mid W}$ is a collection of $F_{\mid W}$. We derive a second collection from 
$\mathcal{H}$ by
\[
\mathcal{H}_W = 
\{ A_{W \cup V} \cap compl(\bigcup_{A \in\mathcal{H}_{\mid W}}A) \cap (V \times Y)  
\mid V \subset X \backslash W, V \ne \emptyset \}.
\]                         
$\mathcal{H}_W$ is a collection of $F_{\mid (X \backslash W)}$. Please note, in the definition 
$\mathcal{H}_W$ and also in Definition \ref{D:distributed1} we consider the collection 
$\mathcal{H}_{\mid W}$ (with sets $A_V$, $V \subset W$) and sets $A_{W \cup V}$ (where
$W \cup V$ contains $W$). We do not consider sets $A_V$ where $V \subset X \backslash W$.
These definitions reflect ideas used by Halmos and Vaughan \cite{HV} in their proof of the 
marriage theorem.

Let $\mathcal{H} = \{ A_V \subset G(F_{\mid V}) \mid V \subset X, V \ne \emptyset \}$ be a
collection of $F$ and let $W \subset X$ be a nonempty set. In the next step we consider a submapping 
of $F_{\mid W}$ and show that $\mathcal{H}_{\mid W}$ is a collection of this submapping.
We define a submapping $F^{\mathcal{H}_W}$ of $F_{\mid W}$ by
\[
G(F^{\mathcal{H}_W}) = G(F_{\mid W}) \cap compl(\bigcup_{A \in \mathcal{H}_W}A).
\] 

\begin{lemma} \label{L:HmidW}
Let $\mathcal{H}$ be a collection of $F$ and let $W \subset X$ be a nonempty set. 
$\mathcal{H}_{\mid W}$ is a collection of $F^{\mathcal{H}_W}$.
\end{lemma}
\begin{proof}
By definition $\mathcal{H}_{\mid W}$ is a collection of $F_{\mid W}$. From the definition of 
$\mathcal{H}_W$ we obtain the relation 
$\bigcup_{A \in\mathcal{H}_W}A \subset compl(\bigcup_{A \in\mathcal{H}_{\mid W}}A)$. 
Using the definition of 
$F^{\mathcal{H}_W}$
\[
\bigcup_{A \in\mathcal{H}_{\mid W}}A
\subset compl(compl(\bigcup_{A \in\mathcal{H}_{\mid W}}A))
\subset compl(\bigcup_{A \in \mathcal{H}_W}A)
\subset G(F^{\mathcal{H}_W}),
\]
i.e., $\mathcal{H}_{\mid W}$ is a collection of $F^{\mathcal{H}_W}$.
\end{proof}

We consider the complement mapping $F_{W, G(F^{\mathcal{H}_W})}$ as a submapping of 
$F_{\mid (X \backslash W)}$ and show that $\mathcal{H}_W$ is a collection of this submapping.

\begin{lemma} \label{L:H_W}
Let $\mathcal{H}$ be a collection of $F$ and let $W \subset X$ be a nonempty set. $\mathcal{H}_W$ 
is a collection of the complement mapping $F_{W, G(F^{\mathcal{H}_W})}$.
\end{lemma}
\begin{proof}
Using the definition of $F^{\mathcal{H}_W}$,
\[
\bigcup_{A \in \mathcal{H}_W}A
\subset compl(compl(\bigcup_{A \in \mathcal{H}_W}A))
\subset compl(G(F^{\mathcal{H}_W})).
\]
By definition of the complement mapping 
\[
A 
\subset G(F_{\mid (X \backslash W)}) \cap compl(G(F^{\mathcal{H}_W}))
= G(F_{W,G(F^{\mathcal{H}_W})}) 
\]
for each $A \in \mathcal{H}_W$.
\end{proof}

At this step we introduce an additional condition on collections and call them Hall collections.

\begin{definition} \label{D:collection}
A collection $\mathcal{H}=\{ A_V \subset G(F_{\mid V}) \mid V \subset X, V \ne \emptyset \}$ 
of $F$ is called a Hall collection of $F$ if $A_V$ is a disparate set and $\sharp A_V \ge \sharp V$ for 
each nonempty set $V \subset X$.
\end{definition}

Several authors proposed definitions which are equivalent to Hall collections. Hilton and Johnson 
\cite{HJ2}) considered the maximum size of a transversal. Cropper \cite{Crop} considered the size 
of the largest independent set of vertices and Bobga et al. \cite{Bob2} considered the vertex independence 
number. These authors observed that their condition is necessary but not sufficient for the existence 
of a proper coloring.

We consider three examples depicted in Fig. \ref{Fig2}. The set $X$ consists in all cases of the 
vertices $\{1, 2, 3, 4\}$ and are indicated by circled numbers. The edges are drawn in the graphics.
The values of the mapping $F$ are drawn close to the vertices. In all cases we choose 
$W=\{1,3,4\}$, which is a generalized critical set of $F$ (compare Definition \ref{D:critical}).

The most right example is an extraction and relabeling of the Sudoku cells $\{(1,1), (2,2), (3,2), (8,2)\}$
(in the notation $(row,column)$) of Provan \cite[Table 2]{Pro}. The other examples are modifications of 
this example.

\begin{example} \label{E:41}
We consider Example a) of Fig. \ref{Fig2} and define a Hall collection $\mathcal{H}$ of $F$ by  \\
$A_{\{1\}}=\{(1,2)\}$, $A_{\{2\}}=\{(2,1)\}$, $A_{\{3\}}=\{(3,2)\}$, $A_{\{4\}}=\{(4,3)\}$,  \\
$A_{\{1,2\}}=\{(1,2),(2,1)\}$, $A_{\{1,3\}}=\{(1,3),(3,1)\}$, $A_{\{1,4\}}=\{(1,2),(4,1)\}$,  \\
$A_{\{2,3\}}=\{(2,1),(3,2)\}$, $A_{\{2,4\}}=\{(2,1),(4,3)\}$, $A_{\{3,4\}}=\{(3,2),(4,3)\}$,  \\
$A_{\{1,2,3\}}=\{(1,3),(2,1),(3,2)\}$, $A_{\{1,2,4\}}=\{(1,2),(2,1),(4,3)\}$,   \\
$A_{\{1,3,4\}}=\{(1,3),(3,1),(4,1)\}$, $A_{\{2,3,4\}}=\{(2,1),(3,2),(4,3)\}$ and  \\
$A_{\{1,2,3,4\}}=\{(1,2),(1,3),(3,1),(4,1)\}$. Then $\mathcal{H}_W=\{\emptyset\}$ is a collection 
of $F_{\mid (X \backslash W)}=F_{\mid \{2\}}$, but $\mathcal{H}_W$ is not a Hall collection of 
$F_{\mid (X \backslash W)}$. It can be easily verified that $F$ does not admit a disparate selection.
\end{example}

\begin{figure}
\includegraphics{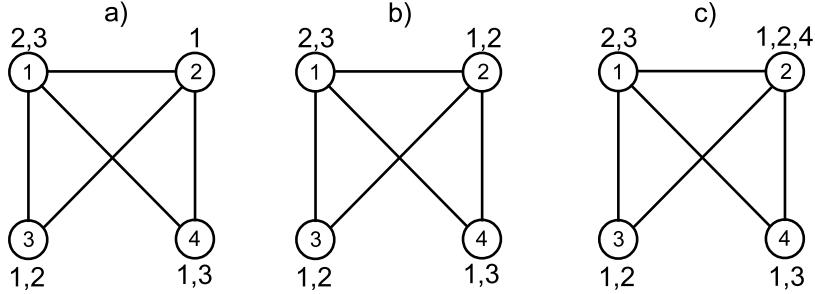}
\caption{Hall Collections} 
\label{Fig2}
\end{figure}

\begin{example} \label{E:42}
In Example b) $F$ admits a disparate selection $s$ given by $s(1)=\{3\}$, $s(2)=\{2\}$, 
$s(3)=\{1\}$, $s(4)=\{1\}$. This selection induces a Hall collection $\mathcal{H}^s$ according 
to Section \ref{S:necessity}. $\mathcal{H}^s_W=\{\{(2,2)\}\}$ and 
$G(F^{\mathcal{H}^s_W})=G(F_{\mid W}) \backslash \{(3,2)\}$.
\end{example}

\begin{example} \label{E:43}
In Example c) $F$ admits at least two disparate selections $s_1$ and $s_2$ given by 
$s_1(1)=\{3\}$, $s_1(2)=\{2\}$, $s_1(3)=\{1\}$, $s_1(4)=\{1\}$ and $s_2(1)=\{2\}$, 
$s_2(2)=\{4\}$, $s_2(3)=\{1\}$, $s_2(4)=\{1\}$. Again these selections induce Hall collections 
$\mathcal{H}^{s_1}$ and $\mathcal{H}^{s_2}$. $\mathcal{H}^{s_1}_W=\{\{(2,2)\}\}$, 
$\mathcal{H}^{s_2}_W=\{\{(2,4)\}\}$, 
$G(F^{\mathcal{H}^{s_1}_W})=G(F_{\mid W}) \backslash \{(3,2)\}$ and
$G(F^{\mathcal{H}^{s_2}_W})=G(F_{\mid W})$.
\end{example}

It is immediately clear that $\mathcal{H}_{\mid W}$ is a Hall collection of $F_{\mid W}$ and
that $\mathcal{H}_W$ is a collection of disparate sets if $\mathcal{H}$ is a Hall collection of $F$. 
At this moment we do not know if $\mathcal{H}_W$ is also a Hall collection. 

In the last step of this section we implement a condition on $\mathcal{H}$ such that $\mathcal{H}_W$ 
is a Hall collection.

\begin{definition} \label{D:distributed1}
Let $W \subset X$ be a nonempty set. A collection 
$\mathcal{H}=\{A_V \subset G(F_{\mid V}) \mid V \subset X, V \ne \emptyset \}$ of $F$ is called 
$W$-distributed if 
\[
\sharp (A_{W \cup V} 
\cap ((W \times Y) \cup (hull(\bigcup_{A \in\mathcal{H}_{\mid W}}A) \cap (V \times Y))))
\le \sharp W
\]
for each nonempty set $V \subset X \backslash W$. 
\end{definition}

\begin{example} \label{E:44}
We consider the Hall collection $\mathcal{H}$ of Example \ref{E:41} and observe that $\mathcal{H}$
is not $W$-distributed. Choose $V= X \backslash W = \{2\}$, consider 
$A_{W \cup V} = A_{\{1,2,3,4\}} \subset W \times Y$ and
$hull(\bigcup_{A \in \mathcal{H}_{\mid W}} A) = G(F)$.
The expression of Definition \ref{D:distributed1} reads as $\sharp A_{\{1,2,3,4\}} = 4 > 3 = \sharp W$.
In this example there does not exist a $W$-distributed Hall collection of $F$.
\end{example}

Under the assumption of Definition \ref{D:distributed1}, it is possible to show that $\mathcal{H}_W$ 
is a Hall collection.

\begin{lemma} \label{L:distributed1}
Let $\mathcal{H}$ be a Hall collection of $F$, let $W \subset X$ be a nonempty set and let 
$\mathcal{H}$ be $W$-distributed. $\mathcal{H}_W$ is a Hall collection of the complement mapping 
$F_{W, G(F^{\mathcal{H}_W})}$.
\end{lemma}
\begin{proof}
Let $\mathcal{H}=\{A_V \subset G(F_{\mid V}) \mid V \subset X, V \ne \emptyset \}$ and let 
$V \subset X \backslash W$ be a nonempty set. Using Lemma \ref{L:H_W}
\[
A_{W \cup V} \cap compl(\bigcup_{A \in\mathcal{H}_{\mid W}}A) \cap (V \times Y) 
\subset G(F_{W, G(F^{\mathcal{H}_W})})
\]
is a disparate set and using Lemma \ref{L:decomposition1} with 
$Z=\bigcup_{A \in\mathcal{H}_{\mid W}}A$ yields
\begin{align*}
&\sharp (A_{W \cup V} 
\cap compl(\bigcup_{A \in\mathcal{H}_{\mid W}}A)
\cap (V \times Y))  \\
&= \sharp A_{W \cup V} - 
\sharp (A_{W \cup V} 
\cap ((W \times Y) 
\cup (hull(\bigcup_{A \in\mathcal{H}_{\mid W}}A)
\cap (V \times Y))))  \\
&\ge  \sharp (W \cup V) - \sharp W  \\
&= \sharp V,
\end{align*}
i.e., $\mathcal{H}_W$ is a Hall collection.
\end{proof}

                                        %
                                        %
\section{Cascades} \label{S:cascades}
We extend the notation and results of Section \ref{S:collection} to cascades. Let $W_0 \subset X$. 
For any tuple $(W_1, \ldots, W_k)$, $k \in \mathbb{N}$, where $W_i \subset W_0$ for $i=1, \ldots, k$, 
we define the complementary tuple $(W^c_1, \ldots, W^c_k)$ of $(W_1, \ldots, W_k)$ in $W_0$ 
inductively by $W^c_0 = W_0$ and 
\[
W^c_i = 
\begin{cases}
W_{i-1} \backslash W_i     & \mbox{if } W_i \subset W_{i-1},  \\
W^c_{i-1} \backslash W_i & \mbox{if } W_i \not\subset W_{i-1},
\end{cases}
\]
for $i=1, \ldots, k$.

\begin{definition} \label{D:cascade}
A tuple $(W_1, \ldots, W_k)$, $k \in \mathbb{N}$, is called a cascade in $W_0 \subset X$ if 
$W_i \subset W_{i-1}$ or $W_i \subset W^c_{i-1}$ for $i=1, \ldots, k$.
\end{definition}

Please note, the complementary tuple $(W^c_1, \ldots, W^c_k)$ of a cascade $(W_1, \ldots, W_k)$ 
in $W_0$ is also a cascade in $W_0$.

If $(W_1, \ldots, W_k)$, $k \in \mathbb{N}$, is a cascade in $X$ with complementary tuple 
$(W^c_1, \ldots, W^c_k)$, $(W_2, \ldots, W_k)$ is a cascade in $W_1$ or $W^c_1$.
Let $W_0 \subset X$. If $(W_1, \ldots, W_k)$, $k \in \mathbb{N}$, is a cascade in $W_0$ or in
$X \backslash W_0$, $(W_0, W_1, \ldots, W_k)$ is a cascade in $X$.

Let $(W_1, \ldots, W_k)$, $k \in \mathbb{N}$, be a cascade in $W_0 \subset X$ with complementary
cascade $(W^c_1, \ldots, W^c_k)$. We define the predecessor function by 
\[
pre(i) = 
\begin{cases}
W_{i-1}     & \mbox{if } W_i \subset W_{i-1},  \\
W^c_{i-1} & \mbox{if } W_i \subset W^c_{i-1},
\end{cases}
\]
for $i=1, \ldots, k$. 

\begin{lemma} \label{L:pre}
Let $(W_1, \ldots, W_k)$, $k \in \mathbb{N}$, be a cascade in $W_0 \subset X$ with complementary
tuple $(W^c_1, \ldots, W^c_k)$. $pre(i) = W_i \cup W^c_i$ for $i=1, \ldots, k$.
\end{lemma}
\begin{proof}
Let $1\le i \le k$. We distinguish two cases, use the definition of cascades and of the complementary 
tuple.  \\
Case 1: $W_i \subset W_{i-1}$.  \\
$pre(i) = W_{i-1} = W_i \cup (W_{i-1} \backslash W_i) = W_i \cup W^c_i$.  \\
Case 2: $W_i \subset W^c_{i-1}$. \\
$pre(i) = W^c_{i-1} = W_i \cup (W^c_{i-1} \backslash W_i) = W_i \cup W^c_i$.
\end{proof}

Let $\mathcal{H}$ be a collection of $F$ and let $(W_1, \ldots, W_k)$, $k \in \mathbb{N}$, be a 
cascade in $W_0=X$ with complementary tuple $(W^c_1, \ldots, W^c_k)$. Denote 
$\mathcal{H}(W_1, \ldots, W_0) = \mathcal{H}$ and $F(W_1, \ldots, W_0) = F$. We define inductively
\[
\mathcal{H}(W_1, \ldots, W_i) = 
\begin{cases}
\mathcal{H}(W_1, \ldots, W_{i-1})_{\mid W_i} & \mbox{if } W_{i+1} \subset W_i,  \\
\mathcal{H}(W_1, \ldots, W_{i-1})_{W_i} & \mbox{if } W_{i+1} \subset W^c_i,
\end{cases}
\]
the abbreviation
$\Tilde{F}(W_1, \ldots, W_i) = F(W_1, \ldots, W_{i-1})^{\mathcal{H}(W_1, \ldots, W_{i-1})_{W_i}}$ and 
\[
F(W_1, \ldots, W_i) =
\begin{cases}
\Tilde{F}(W_1, \ldots, W_i)                                               & \mbox{if } W_{i+1} \subset W_i,  \\
F(W_1, \ldots, W_{i-1})_{W_i,G(\Tilde{F}(W_1, \ldots, W_i))}  & \mbox{if } W_{i+1} \subset W^c_i
\end{cases}
\]
for $i=1, \ldots, k-1$.

Please note, $F(W_1, \ldots, W_i)$ depends on the Hall collection $\mathcal{H}$, but this dependance 
is not stated explicitly and we do not define the expressions $\mathcal{H}(W_1, \ldots, W_k)$ and 
$F(W_1, \ldots, W_k)$.


\begin{lemma} \label{L:cascade1}
Let $\mathcal{H}$ be a collection of $F$ and let $(W_1, \ldots, W_k)$, $k \in \mathbb{N}$, be a 
cascade in $W_0=X$.
\begin{enumerate}[(i)] 
\item $\mathcal{H}(W_1, \ldots, W_{i-1})$ is a collection of $F(W_1, \ldots, W_{i-1})$ and
\item $F(W_1, \ldots, W_{i-1})$ is defined on $pre(i)$,
\end{enumerate}
for $i=1, \ldots, k$.
\end{lemma}
\begin{proof}
Let $(W^c_1, \ldots, W^c_k)$ be the complementary tuple of $(W_1, \ldots, W_k)$. We prove 
the statements by induction on $i=1, \ldots, k$. The statements are true for $i=1$. Let $2\le i \le k$. 
\begin{enumerate}[(i)] 
\item By induction hypothesis $\mathcal{H}(W_1, \ldots, W_{i-2})$ is a collection of 
$F(W_1,$ $\ldots, W_{i-2})$. The statement follows from Lemmas \ref{L:HmidW} and \ref{L:H_W}
and the definitions of $\mathcal{H}(W_1, \ldots, W_{i-1})$ and $F(W_1, \ldots, W_{i-1})$.
\item By induction hypothesis the mapping $F(W_1, \ldots, W_{i-2})$ is defined on $pre(i-1)$ and by 
Lemma \ref{L:pre}, $pre(i-1) = W_{i-1} \cup W^c_{i-1}$. If $W_{i} \subset W_{i-1}$, 
$F(W_1, \ldots, W_{i-1})$ is defined on $W_{i-1} = pre(i)$ and if $W_{i} \subset W^c_{i-1}$, 
$F(W_1, \ldots, W_{i-1})$ is defined on $W^c_{i-1} = pre(i)$.
\end{enumerate}
\end{proof}

\begin{lemma} \label{L:cascade2}
Let $\mathcal{H}$ be a collection of $F$ and let $(W_1, \ldots, W_k)$, $k \in \mathbb{N}$, be a 
cascade in $X$.
\begin{enumerate}[(i)] 
\item $\mathcal{H}(W_1, \ldots, W_i) = \mathcal{H}(W_1)(W_2, \ldots, W_i)$ and
\item $F(W_1, \ldots, W_i) = F(W_1)(W_2, \ldots, W_i)$
\end{enumerate}
for $i=1, \ldots, k-1$.
\end{lemma}
\begin{proof}
Let $(W^c_1, \ldots, W^c_k)$ be the complementary tuple of $(W_1, \ldots, W_k)$. Using 
Lemma \ref{L:cascade1}, $\mathcal{H}(W_1)$ is a collection of $F(W_1)$.
We prove the statements by induction on $i=1, \ldots, k-1$. The statements are true for $i=1$. 
Let $2\le i \le k-1$. We distinguish two cases.  \\
Case 1: $W_{i+1} \subset W_i$.  
\begin{enumerate}[(i)] 
\item Using the induction hypothesis
\begin{align*}                                                   
& \mathcal{H}(W_1, \ldots, W_i) = \mathcal{H}(W_1, \ldots, W_{i-1})_{\mid W_i}  \\
& = \mathcal{H}(W_1)(W_2, \ldots, W_{i-1})_{\mid W_i} 
= \mathcal{H}(W_1)(W_2, \ldots, W_i).
\end{align*}
\item Using the induction hypothesis and ``(i)"
\begin{align*}                                                   
& F(W_1, \ldots, W_i) 
= F(W_1, \ldots, W_{i-1})^{\mathcal{H}(W_1, \ldots, W_{i-1})_{W_i}}  \\
& = F(W_1)(W_2, \ldots, W_{i-1})^{\mathcal{H}(W_1)(W_2, \ldots, W_{i-1})_{W_i}}
= F(W_1)(W_2, \ldots, W_i).
\end{align*}
\end{enumerate}
Case 2: $W_{i+1} \subset W^c_i$.  
\begin{enumerate}[(i)] 
\item Using the induction hypothesis
\begin{align*}                                                   
& \mathcal{H}(W_1, \ldots, W_i) = \mathcal{H}(W_1, \ldots, W_{i-1})_{W_i}  \\
& = \mathcal{H}(W_1)(W_2, \ldots, W_{i-1})_{W_i} = \mathcal{H}(W_1)(W_2, \ldots, W_i).
\end{align*}
\item Using the induction hypothesis and ``(i)"
\begin{align*}                                                   
& F(W_1, \ldots, W_i) 
= F(W_1, \ldots, W_{i-1})_{W_i,G(F(W_1, \ldots, W_{i-1})^{\mathcal{H}(W_1, \ldots, W_{i-1})_{W_i}})}  \\
& = F(W_1)(W_2, \ldots, W_{i-1})
_{W_i,G(F(W_1)(W_2, \ldots, W_{i-1})^{\mathcal{H}(W_1)(W_2, \ldots, W_{i-1})_{W_i}})}  \\
& = F(W_1)(W_2, \ldots, W_i).
\end{align*}
\end{enumerate}
\end{proof}

In the remaining part of this section we consider a property of Hall collections of $F$.

\begin{lemma} \label{L:cascadeHall}  
Let $\mathcal{H}$ be a Hall collection of $F$ and let $(W_1, \ldots, W_k)$, $k \in \mathbb{N}$, be a 
cascade in $W_0=X$ such that $\mathcal{H}(W_1, \ldots , W_{i-1})$ is $W_i$-distributed for $i=1, \ldots, k$. 
$\mathcal{H}(W_1, \ldots , W_i)$ is a Hall collection of $F(W_1, \ldots , W_i)$ for $i=1, \ldots, k-1$. 
\end{lemma}
\begin{proof}
Let $(W^c_1, \ldots, W^c_k)$ be the complementary tuple of $(W_1, \ldots, W_k)$. Using 
Lemma \ref{L:cascade1}, $\mathcal{H}(W_1)$ is a collection of $F(W_1)$.
We prove the statement by induction on $i=1, \ldots, k-1$. Let $i=1$. $\mathcal{H}(W_1)$ 
is a Hall collection of $F(W_1)$ by Lemma \ref{L:HmidW} and \ref{L:distributed1}. Let $2 \le i \le k-1$.  
We distinguish two cases.  \\
Case 1: $W_{i+1} \subset W_i$.  \\ 
Using the induction hypothesis and Lemma \ref{L:HmidW},
$\mathcal{H}(W_1, \ldots , W_{i-1})_{\mid W_i}$ is a Hall collection of 
$F(W_1, \ldots , W_{i-1})^{\mathcal{H}(W_1, \ldots , W_{i-1})_{W_i}}$. The statement follows
from $W_{i+1} \subset W_i$ and the definition of $\mathcal{H}(W_1, \ldots , W_i)$ and 
$F(W_1, \ldots , W_i)$.  \\
Case 2: $W_{i+1} \subset W^c_i$.  \\ 
Using the induction hypothesis and Lemma \ref{L:distributed1},
$\mathcal{H}(W_1, \ldots , W_{i-1})_{W_i}$ is a Hall collection of
$F(W_1, \ldots , W_{i-1})_{W_i,G(F(W_1, \ldots , W_{i-1})^{\mathcal{H}(W_1, \ldots , W_{i-1})_{W_i}})}$.
The statement follows from $W_{i+1} \subset W^c_i$ and the definition of $\mathcal{H}(W_1, \ldots , W_i)$ and 
$F(W_1, \ldots , W_i)$.  
\end{proof}

                                        %
                                        %
\section{Generalized Marriage Theorem} \label{S:main}
This section contains the main result of this paper. We prove a sufficient and necessary condition 
on $F$ for the existence of a disparate selection. This theorem is called a generalized marriage
theorem and the proof is a generalization of the proof of Halmos and Vaughan \cite{HV}
of the marriage theorem.

Before we are able to state this theorem we have to introduce the generalized Hall condition in an 
inductive process. Within this process we define generalized critical sets.

\begin{definition} \label{D:Hall1}
Let $\sharp X = 1$. $F$ is called to satisfy the generalized Hall condition if $G(F) \ne \emptyset$.
\end{definition}

This definition describes the initialization of the generalized Hall condition. By induction hypothesis 
we defined the generalized Hall condition for sets $X^\prime$ with $\sharp X^\prime < \sharp X$.
We need some more definitions before we define the generalized Hall condition for sets $X$ of 
arbitrary size.

\begin{definition} \label{D:critical}
A nonempty set $W \subset X$ is called a generalized critical set of $F$ if there exists 
$(x,y) \in (X \backslash W) \times Y$ such that the complement mapping 
${(F_{\{x\}, \{(x,y)\}})}_{\mid W}$ does not satisfy the generalized Hall condition.
\end{definition}

In particular $\sharp W < \sharp X$ in Definition \ref{D:critical} and the generalized Hall condition
of the mapping ${(F_{\{x\}, \{(x,y)\}})}_{\mid W}$ is defined by induction hypothesis.

Critical sets had been used by Halmos and Vaughan \cite{HV} in their proof of the marriage theorem 
without naming them explicitly. Easterfield \cite{Eas} used critical sets in an algorithm determining 
a minimal sum in an $n \times n$-array of real numbers and called them ``exactly adjusted". 
Everett and Whaples \cite{EW} proved a generalization of Hall's theorem and called them ``perfect". 
The term ``critical" had been introduced by M. Hall \cite{HalM2} and he considered critical blocks. 
Schrijver \cite{Sch} called them ``tight" sets. 

The consideration of critical subsets had been used by Crook \cite{Cro} and Provan \cite{Pro} in their 
description of a strategy solving Sudoku puzzles. Crook used the term ``preemptive" set and Provan 
used the term ``pigeon-hole rule". 

\begin{definition} \label{D:critical2}  
Let $\mathcal{H}$ be a Hall collection of $F$. A cascade $(W_1, \ldots,$ $W_k)$, $k \in \mathbb{N}$, 
in $W_0=X$ is called $(F,\mathcal{H})$-critical if $W_i$ is a generalized critical set of 
$F(W_1, \ldots, W_{i-1})$ for $i=1, \ldots, k$.
\end{definition}

A generalized critical set $W \subset X$ of $F$ is called minimal if there does not exist a subset 
$W^\prime \subset W$, $W^\prime \ne W$, such that $W^\prime$ is a generalized critical set of $F$.
Please note, $F$ admits a primitive generalized critical set if and only if $F$ admits a generalized critical 
set.

\begin{definition} \label{D:primitive1}
Let $\mathcal{H}$ be a Hall collection of $F$. An $(F,\mathcal{H})$-critical cascade $(W_1, \ldots, W_k)$, 
$k \in \mathbb{N}$, in $W_0=X$ is called primitive if $W_i$ is a minimal generalized critical set of 
$F(W_1, \ldots, W_{i-1})$ for $i=1, \ldots, k$.
\end{definition}

Using the notation of primitive critical cascades we are able to introduce a ``distributed"-condition 
on a Hall collection.

\begin{definition} \label{D:distributed2}
A Hall collection $\mathcal{H}$ of $F$ is called distributed if the collection 
$\mathcal{H}(W_1, \ldots, W_{i-1})$ of $F(W_1, \ldots, W_{i-1})$ is $W_i$-distributed for $i=1, \ldots, k$ 
and each primitive $(F,\mathcal{H})$-critical cascade $(W_1, \ldots, W_k)$, $k \in \mathbb{N}$, in $W_0=X$.
\end{definition}

As a consequence of Lemma \ref{L:cascadeHall}, $\mathcal{H}(W_1, \ldots, W_i)$ is a Hall collection of
$F(W_1, \ldots, W_i)$ for $i=1, \ldots, k-1$, each primitive $(F,\mathcal{H})$-critical cascade 
$(W_1, \ldots, W_k)$ in $X$ and each distributed Hall collection $\mathcal{H}$ of $F$.

\begin{definition} \label{D:Hall2}
Let $\sharp X > 1$. $F$ is called to satisfy the generalized Hall condition if $F$ admits a distributed 
Hall collection.
\end{definition}

Definition \ref{D:Hall2} completes the inductive step in the definition of the generalized Hall condition 
and we are able to state our main result.

\begin{theorem}[Generalized Marriage Theorem] \label{T:main}
The following statements are equivalent: 
\begin{enumerate}[(i)] 
\item $F$ satisfies the generalized Hall condition.
\item $F$ admits a disparate selection.
\end{enumerate}
\end{theorem}
\begin{proof}
The proof of this theorem requires several lemmas in both directions, which are collected in the next
two sections. The implication ``(i) $\Rightarrow$ (ii)" is contained in Lemma \ref{L:79}. The 
implication ``(ii) $\Rightarrow$ (i)" is contained in Lemma \ref{L:87}.
\end{proof}

                                        %
                                        %
\section{Sufficiency} \label{S:sufficiency}
This section contains the lemmas, which are used in the sufficiency part of the generalized marriage 
theorem. The results of this section are subsumed in Lemma \ref{L:79}.

\begin{lemma} \label{L:74}
Let $F$ satisfy the generalized Hall condition and assume $F$ does not admit a generalized critical set.
The complement mapping $F_{\{x\}, \{(x,y)\}}$ satisfies the generalized Hall condition for each 
$(x,y) \in G(F)$.
\end{lemma}
\begin{proof}
Let $(x,y) \in G(F)$. Using the assumption, $X \backslash \{x\}$ is not a generalized critical set of $F$,
i.e., $F_{\{x\}, \{(x,y)\}}$ satisfies the generalized Hall condition (compare Definition \ref{D:critical}).
\end{proof}

In their proof of the marriage theorem Halmos and Vaughan \cite{HV} distinguish the cases where 
$F$ admits a critical set or not. In our context the case where $F$ does not admit a critical set is 
treated in Lemma \ref{L:74}. The case where $F$ admits a critical set is treated in Lemma \ref{L:78}.
We need a preparatory lemma.

\begin{lemma} \label{L:7prep}
Let $\mathcal{H}$ be a Hall collection of $F$. Let $W_1 \subset X$ be a minimal generalized 
critical set of $F$ and let $(W_2, \ldots, W_k)$, $k \in \mathbb{N}$, $k \ge 2$, be a primitive 
$(F(W_1),\mathcal{H}(W_1))$-critical cascade in $W_1$ or in $X \backslash W_1$. 
$(W_1, W_2, \ldots, W_k)$ is a primitive $(F,\mathcal{H})$-critical cascade in $X$.
\end{lemma}
\begin{proof}
$(W_1, W_2, \ldots, W_k)$ is a cascade in $W_0 = X$. By assumption $W_1$ is a minimal generalized 
critical set of $F(W_1, \ldots, W_0)=F$. Using the assumption and Lemma \ref{L:cascade2}, $W_i$ is a 
minimal generalized critical set of $F(W_1, \ldots, W_{i-1})=F(W_1)(W_2, \ldots, W_{i-1})$ for 
$i=2, \ldots, k$.  
\end{proof}

\begin{lemma} \label{L:78}
Let $F$ satisfy the generalized Hall condition and let $W \subset X$ be a minimal generalized critical 
set of $F$. There exists a submapping $F^\prime$ of $F_{\mid W}$ such that $F^\prime$ satisfies 
the generalized Hall condition and the complement mapping $F_{W, G(F^\prime)}$ satisfies the 
generalized Hall condition.
\end{lemma}
\begin{proof}
By assumption $F$ admits a distributed Hall collection $\mathcal{H}$. We define a
submapping of $F_{\mid W}$ by $F^\prime = F^{\mathcal{H}_W}$ and define $W_1 = W$.

Using Lemma \ref{L:HmidW}, $\mathcal{H}_{\mid W}$ is a Hall collection of $F^{\mathcal{H}_W}$. 
Let $(W_2, \ldots, W_k)$, $k \in \mathbb{N}$, $k \ge 2$, be a primitive 
$(F^{\mathcal{H}_W},\mathcal{H}_{\mid W})$-critical cascade in $W$.  
Using Lemma \ref{L:7prep}, $(W_1, \ldots, W_k)$ is a primitive $(F,\mathcal{H})$-critical cascade in 
$X$. Using Lemma \ref{L:cascade2} the collection
\[
\mathcal{H}_{\mid W_1}(W_2, \ldots, W_{i-1})=\mathcal{H}(W_1, \ldots, W_{i-1})
\]
of $F^{\mathcal{H}_{W_1}}(W_2, \ldots, W_{i-1})=F(W_1, \ldots, W_{i-1})$ is $W_i$-distributed 
for $i=2, \ldots, k$, i.e., $F^\prime=F^{\mathcal{H}_{W}}=F^{\mathcal{H}_{W_1}}$ satisfies the 
generalized Hall condition.

By Lemma \ref{L:distributed1}, $\mathcal{H}_W$ is a Hall collection of $F_{W,G(F^\prime)}$.
Let $(W_2, \ldots, W_k)$, $k \in \mathbb{N}$, $k \ge 2$, be a primitive 
$(F_{W,G(F^\prime)},\mathcal{H}_W)$-critical cascade in $X \backslash W$. 
Using Lemma \ref{L:7prep}, $(W_1, \ldots, W_k)$ is a primitive $(F,\mathcal{H})$-critical cascade in 
$X$. Using Lemma \ref{L:cascade2} the collection 
\[
\mathcal{H}_{W_1}(W_2, \ldots, W_{i-1}) = \mathcal{H}(W_1, \ldots, W_{i-1})  
\]
of $F_{W_1,G(F^\prime)}(W_2, \ldots, W_{i-1})=F(W_1, \ldots, W_{i-1})$ is $W_i$-distributed for 
$i=2, \ldots, k$, i.e., $F_{W,G(F^\prime)}=F_{W_1,G(F^\prime)}$ satisfies the generalized Hall condition.
\end{proof}

Lemma \ref{L:79} shows the sufficiency part of the generalized marriage theorem.

\begin{lemma} \label{L:79}
Let $F$ satisfy the generalized Hall condition. $F$ admits a disparate selection.
\end{lemma}
\begin{proof}
We prove the statement by induction on the number $\sharp X$ of elements in $X$. The statement is 
true for $\sharp X = 1$ (compare Definition \ref{D:Hall1}).

Let $\sharp X > 1$. By induction hypothesis the statement is true for subsets 
$X^\prime \subset X$ with $\sharp X^\prime < \sharp X$. We distinguish two cases.  \\
Case 1: $F$ does not admit a generalized critical set.  \\
Choose $(x,y) \in G(F)$ and apply Lemma \ref{L:74}. Since $\sharp (X \backslash \{x\}) < \sharp X$ 
we apply the induction hypothesis and there exists a disparate selection $s$ of $F_{\{x\},\{(x,y)\}}$.
The statement follows from Lemma \ref{L:complement1}.  \\
Case 2: $F$ admits a generalized critical set.  \\
Then $F$ admits a minimal generalized critical set $W \subset X$ and  we know $W \ne \emptyset$ 
and $W \ne X$. Using Lemma \ref{L:78} there exists a submapping $F^\prime$ of $F_{\mid W}$ 
such that $F^\prime$ satisfies the generalized Hall condition and $F_{W, G(F^\prime)}$ satisfies 
the generalized Hall condition.

By induction hypothesis there exist disparate selections $s_1$ of $F^\prime$ and $s_2$ of 
$F_{W, G(F^\prime)}$. We define a selection $s$ of $F$ by $s_{\mid W}=s_1$ and 
$s_{\mid (X \backslash W)}=s_2$. Using Lemma \ref{L:complement2} with $Z=G(F^\prime)$, 
$s$ is disparate.
\end{proof}

This completes the sufficiency part of the generalized marriage theorem.

                                        %
                                        %
\section{Necessity} \label{S:necessity}
This section contains the lemmas, which are used in the necessity part of the generalized marriage 
theorem. The results of this section are subsumed in Lemma \ref{L:87}.

Any selection $s$ of $F$ induces a collection of $F$ by
\[
\mathcal{H}^s = \{G(s_{\mid W}) \subset G(F_{\mid W}) \mid W \subset X, W \ne \emptyset\}.
\]

We collect some elementary properties of disparate selections and the induced collection $\mathcal{H}^s$.

\begin{lemma} \label{L:81}
Let $s$ be a disparate selection of $F$. $\mathcal{H}^s$ is a Hall collection of $F$.
\end{lemma}
\begin{proof}
The set $G(s_{\mid W}) \subset G(s)$ is a disparate set and $\sharp G(s_{\mid W}) = \sharp W$
for each $W \subset X$.
\end{proof}

\begin{lemma} \label{L:82}
Let $s$ be a disparate selection of $F$ and let $W \subset X$ be a nonempty set. 
\begin{enumerate}[(i)] 
\item $\mathcal{H}^s_{\mid W}=\{G(s_{\mid V}) \mid V \subset W, V \ne \emptyset \}$.
\item $\bigcup_{A \in \mathcal{H}^s_{\mid W}}A=G(s_{\mid W})$.
\item $\mathcal{H}^s_W=\{G(s_{\mid V}) \mid V \subset X \backslash W, V \ne \emptyset \}$.
\item $\bigcup_{A \in \mathcal{H}^s_W}A=G(s_{\mid (X \backslash W)})$.
\end{enumerate}
\end{lemma}
\begin{proof}
The statements ``(i)" and ``(ii)" are a direct consequence of the definition of $\mathcal{H}^s$.
Using ``(ii)" and since $s$ is a disparate selection
$G(s_{\mid (X \backslash W)}) \subset compl(G(s_{\mid W}))
= compl(\bigcup_{A \in \mathcal{H}^s_{\mid W}}A)$.
We obtain
\begin{align*}
\mathcal{H}^s_W
& = \{G(s_{\mid W \cup V}) \cap compl(\bigcup_{A \in \mathcal{H}^s_{\mid W}}A) \cap (V \times Y)
\mid V \subset X \backslash W, V \ne \emptyset \}  \\
& = \{G(s_{\mid V}) \cap compl(\bigcup_{A \in \mathcal{H}^s_{\mid W}}A)
\mid V \subset X \backslash W, V \ne \emptyset \}  \\
& = \{G(s_{\mid V}) \mid V \subset X \backslash W, V \ne \emptyset \},
\end{align*}
i.e., ``(iii)" and this implies ``(iv)".
\end{proof}

\begin{lemma} \label{L:distributed2}
Let $s$ be a disparate selection of $F$ and let $W \subset X$ be a nonempty set. $\mathcal{H}^s$
is $W$-distributed.
\end{lemma}
\begin{proof}
Let $V \subset X \backslash W$ be a nonempty set. Using Lemma \ref{L:82} ``(ii)"
\begin{align*}
&\sharp (G(s_{\mid W \cup V}) 
\cap ((W \times Y) \cup (hull(\bigcup_{A \in\mathcal{H}^s_{\mid W}}A)
\cap (V \times Y))))  \\
&= \sharp (G(s_{\mid W}) \cup (G(s_{\mid V}) \cap hull(G(s_{\mid W}))))  \\
&= \sharp G(s_{\mid W})  \\
&= \sharp W.  
\end{align*}
\end{proof}

Please note, in the proof of Lemma \ref{L:distributed2} we do not require $W$ to be a generalized 
critical set or a minimal generalized critical set. A disparate selection induces a $W$-distributed Hall 
collection for any nonempty set $W \subset X$.

\begin{lemma} \label{L:83}
Let $s$ be a disparate selection of $F$ and let $W \subset X$ be a nonempty set. 
\begin{enumerate}[(i)] 
\item $\mathcal{H}^{s_{\mid W}}=\mathcal{H}^s_{\mid W}$.
\item $s_{\mid W}$ is a disparate selection of $F^{\mathcal{H}^s_W}$.
\item $\mathcal{H}^{s_{\mid (X \backslash W)}}=\mathcal{H}^s_W$.
\item $s_{\mid (X \backslash W)}$ is a disparate selection of $F_{W,G(F^{\mathcal{H}^s_W})}$.
\end{enumerate}
\end{lemma}
\begin{proof}
``(i)" This statement is clear.  \\
``(ii)" $G(s_{\mid W})$ is a disparate set. Using Lemma \ref{L:82} ``(iv)"
\[
G(s_{\mid W}) \subset compl(G(s_{\mid (X \backslash W)})) = compl(\bigcup_{A \in \mathcal{H}^s_W}A)
\]
and by definition of $F^{\mathcal{H}^s_W}$, $G(s_{\mid W}) \subset G(F^{\mathcal{H}^s_W})$,
i.e., $s_{\mid W}$ is a selection of $F^{\mathcal{H}^s_W}$.  \\
``(iii)" $\mathcal{H}^{s_{\mid (X \backslash W)}} = \{G(s_{\mid V}) \mid V \subset X \backslash W, V \ne \emptyset \}$
and the statement follows from Lemma \ref{L:82} ``(iii)".  \\
``(iv)" $G(s_{\mid (X \backslash W)})$ is a disparate set. By definition of $F^{\mathcal{H}^s_W}$, 
$G(F^{\mathcal{H}^s_W}) \subset compl(\bigcup_{A \in \mathcal{H}^s_W}A)$ and using 
Lemma \ref{L:82} ``(iv)",
\[
G(s_{\mid (X \backslash W)}) 
= \bigcup_{A \in \mathcal{H}^s_W}A 
\subset compl(compl(\bigcup_{A \in \mathcal{H}^s_W}A))
\subset compl(G(F^{\mathcal{H}^s_W})).
\]
By definition of $F_{W,G(F^{\mathcal{H}^s_W})}$, 
$G(s_{\mid (X \backslash W)}) \subset G(F_{W,G(F^{\mathcal{H}^s_W})})$, i.e.,
$s_{\mid (X \backslash W)}$ is a selection of $F_{W,G(F^{\mathcal{H}^s_W})}$.
\end{proof}

\begin{lemma} \label{L:84}
Let $s$ be a disparate selection of $F$ and let $(W_1, \ldots, W_k)$, $k \in \mathbb{N}$, be a 
cascade in $W_0=X$. 
\begin{enumerate}[(i)] 
\item $\mathcal{H}^{s_{\mid pre(i)}}=\mathcal{H}^s(W_1, \ldots, W_{i-1})$ and
\item $s_{\mid pre(i)}$ is a disparate selection of $F(W_1, \ldots, W_{i-1})$
\end{enumerate}
for $i=1, \ldots, k$.
\end{lemma}
\begin{proof}
Let $(W^c_1, \ldots, W^c_k)$ be the complementary tuple of the cascade $(W_1, \ldots, W_k)$.
We prove the statements by induction on $i = 1, \ldots , k$. The statements are true for $i=1$. 
Let $2 \le i \le k$. By induction hypothesis ``(ii)", $s_{\mid pre(i-1)}$ is a disparate selection of 
$F(W_1, \ldots, W_{i-2})$ and we apply Lemma \ref{L:83}, where $pre(i-1) \rightarrow X$,
$W_{i-1} \rightarrow W$, $s_{\mid pre(i-1)} \rightarrow s$ and 
$F(W_1, \ldots, W_{i-2}) \rightarrow F$. We obtain
\begin{enumerate}[(a)] 
\setlength{\itemindent}{.8cm}  
\item $\mathcal{H}^{s_{\mid W_{i-1}}}=\mathcal{H}^{s_{\mid  pre(i-1)}}_{\mid W_{i-1}}$.
\item $s_{\mid W_{i-1}}$ is a disparate selection of 
$F(W_1, \ldots, W_{i-2})^{\mathcal{H}^{s_{\mid  pre(i-1)}}_{W_{i-1}}}$.
\item $\mathcal{H}^{s_{\mid (pre(i-1) \backslash W_{i-1})}}
=\mathcal{H}^{s_{\mid  pre(i-1)}}_{W_{i-1}}$. 
\item $s_{\mid (pre(i-1) \backslash W_{i-1})}$ is a disparate selection of 
\[
F(W_1, \ldots, W_{i-2})
_{W_{i-1},G(F(W_1, \ldots, W_{i-2})^{\mathcal{H}^{s_{\mid  pre(i-1)}}_{W_{i-1}}})}.
\]
\end{enumerate}
We distinguish two cases.   \\
Case 1: $W_i \subset W_{i-1}$.  \\
Using $pre(i)=W_{i-1}$, ``(a)" and induction hypothesis ``(i)" we obtain
\[
\mathcal{H}^{s_{\mid pre(i)}}
=\mathcal{H}^{s_{\mid W_{i-1}}}
=\mathcal{H}^{s_{\mid pre(i-1)}}_{\mid W_{i-1}}
\]
\[
=\mathcal{H}^s(W_1, \ldots, W_{i-2})_{\mid W_{i-1}}
=\mathcal{H}^s(W_1, \ldots, W_{i-1}).
\]
Using $pre(i)=W_{i-1}$, ``(b)" and induction hypothesis ``(i)", $s_{\mid pre(i)}$ is a disparate
selection of 
\[
F(W_1, \ldots, W_{i-1}) 
=F(W_1, \ldots, W_{i-2})^{\mathcal{H}^s(W_1, \ldots, W_{i-2})_{W_{i-1}}}
\]
\[
=F(W_1, \ldots, W_{i-2})^{\mathcal{H}^{s_{\mid pre(i-1)}}_{W_{i-1}}}.
\]
Case 2: $W_i \subset W^c_{i-1}$.  \\
Using $pre(i)=W^c_{i-1}=pre(i-1) \backslash W_{i-1}$, ``(c)" and induction hypothesis ``(i)" 
we obtain
\[
\mathcal{H}^{s_{\mid pre(i)}}
=\mathcal{H}^{s_{\mid (pre(i-1) \backslash W_{i-1})}}
=\mathcal{H}^{s_{\mid pre(i-1)}}_{W_{i-1}}
\]
\[
=\mathcal{H}^s(W_1, \ldots, W_{i-2})_{W_{i-1}}
=\mathcal{H}^s(W_1, \ldots, W_{i-1}).
\]
Using $pre(i)=pre(i-1) \backslash W_{i-1}$, ``(d)" and induction hypothesis ``(i)", $s_{\mid pre(i)}$ 
is a disparate selection of 
\begin{align*}
&F(W_1, \ldots, W_{i-1})  \\
&=F(W_1, \ldots, W_{i-2})
_{W_{i-1},G({F(W_1, \ldots, W_{i-2})^{\mathcal{H}^s(W_1, \ldots, W_{i-2})_{W_{i-1}}}})}  \\
&=F(W_1, \ldots, W_{i-2})
_{W_{i-1},G({F(W_1, \ldots, W_{i-2})^{\mathcal{H}^{s_{\mid pre(i-1)}}_{W_{i-1}}}})}.
\end{align*}
\end{proof}

\begin{lemma} \label{L:86}
Let $s$ be a disparate selection of $F$ and let $(W_1, \ldots, W_k)$, $k \in \mathbb{N}$, be a cascade 
in $W_0=X$. The collection $\mathcal{H}^s(W_1, \ldots, W_{i-1})$ of $F(W_1, \ldots, W_{i-1})$ 
is $W_i$-distributed for $i=1, \ldots, k$.
\end{lemma}
\begin{proof}
Let $1 \le i \le k$. By Lemma \ref{L:84} ``(ii)", $s_{\mid pre(i)}$ is a disparate selection of 
$F(W_1, \ldots, W_{i-1})$. We apply Lemma \ref{L:distributed2} where $pre(i) \rightarrow X$,
$W_i \rightarrow W$, $s_{\mid pre(i)} \rightarrow s$ and $F(W_1, \ldots, W_{i-1}) \rightarrow F$
and obtain $\mathcal{H}^{s_{\mid pre(i)}}$ is $W_i$-distributed. The statement follows from 
Lemma \ref{L:84} ``(i)".
\end{proof}

\begin{lemma} \label{L:87}
Let $s$ be a disparate selection of $F$.  $\mathcal{H}^s$ is a distributed Hall collection 
of $F$.
\end{lemma}
\begin{proof}
Let $(W_1, \ldots, W_k)$, $k \in \mathbb{N}$, be a primitive $(F,\mathcal{H}^s)$-critical cascade 
in $W_0=X$. Using Lemma \ref{L:86}, the collection  $\mathcal{H}^s(W_1, \ldots, W_{i-1})$ 
of $F(W_1, \ldots, W_{i-1})$ is $W_i$-distributed for $i=1, \ldots, k$, i.e., $\mathcal{H}^s$ is a 
distributed Hall collection 
of $F$.
\end{proof}

                                        %
                                        %
\section{The Disparate Kernel} \label{S:kernel}
We introduce a submapping of $F$ called the disparate kernel. This submapping consists of all 
disparate selections of $F$. We define the submapping $F^*$ of $F$ by
\begin{align*}                                                   
F^*(x) = \{ y \in F(x) \mid 
& \mbox{ there exists a disparate selection }  \\
& \mbox{ } s \mbox{ of } F \mbox{ such that } y = s(x) \} 
\end{align*}
for each $x \in X$. We call $F^*$ the disparate kernel of $F$. 

A set-valued mapping $F$ admits a disparate selection if and only if $F^*(x) \ne \emptyset$ for all 
(or one) $x \in X$, i.e., the graph $G(F^\star)$ of $F^\star$ is nonempty if and only if $F$ admits a 
disparate selection. In the case when the graph $(X,E)$ is complete, $F^*$ had been defined in 
\cite[Section 4]{Fis1} and had been called the alldifferent kernel. 

\begin{example} \label{E:kernel.1}
$F$ admits a disparate selection in Example b) of Fig. \ref{Fig2}, described in Example \ref{E:42} 
and this is the only disparate selection of $F$. The disparate kernel is given by 
$G(F^*)=\{(1,3), (2,2), (3,1), (4,1)\}=G(F) \backslash \{(1,2), (2,1), (3,2), (4,3)\}$ and $F^*$ is 
single-valued. 
\end{example}

We state elementary properties of the disparate kernel.

\begin{lemma} \label{L:kernel.1}
\begin{enumerate}[(i)] 
\item $(F^*)^* = F^*$.
\item Let $Q$ be a submapping of $F$. $Q^*$ is a submapping of $F^*$.
\item Let $Q$ be a submapping of $F$ and let $F^*$ be a submapping of $Q$. $Q^*=F^*$.
\end{enumerate}
\end{lemma}
\begin{proof} 
Statement ``(i)" is true, since each disparate selection of $F$ is also a disparate selection of $F^*$.
Statement ``(ii)" is true, since each disparate selection of $Q$ is also a disparate selection of $F$.
``(iii)" Using the statements ``(i)" and ``(ii)" we obtain 
$G(F^*)=G((F^*)^*) \subset G(Q^*) \subset G(F^*)$,
i.e., equality holds.
\end{proof} 

\begin{definition} \label{D:DisparateMappings}
$F$ is called disparate if $G(F)$ is nonempty and $F^* = F$.
\end{definition}

If $F$ admits a disparate selection, $F^*$ is a disparate mapping and $F^*(x) \ne \emptyset$ for each 
$x \in X$. Each disparate mapping admits a disparate selection.

\begin{lemma} \label{L:kernel.2}
Let $F$ admit a disparate selection. The mapping $F^*$ is the largest disparate submapping of $F$.
\end{lemma}
\begin{proof} 
Using Lemma \ref{L:kernel.1} ``(i)", $F^*$ is a disparate mapping and by definition $F^*$ is a 
submapping of $F$. Let $Q$ be a disparate submapping of $F$, let $x \in X$ and $y \in Q(x)$. There 
exists a disparate selection $s$ of $Q$ such that $s(x) = y$. Then $s$ is also a selection of $F$, i.e., 
$y \in F^* (x)$ and $Q$ is a submapping of $F^*$.
\end{proof} 

\begin{definition} \label{D:eliminationPoint}
Let $W \subset X$ be a nonempty set. The elimination points of $W$ are defined by
\begin{align*}                                                   
elim_F(W) = \{ (x,y) \in G(F_{\mid (X \backslash W)}) \mid 
& \mbox{ $(F_{\{x\},\{(x,y)\}})_{\mid W}$ does not satisfy}  \\
& \mbox{ the generalized Hall condition } \}.
\end{align*}
\end{definition}

The basic idea of elimination points is that there is never a disparate selection which goes through
an elimination point $(x,y)$ (compare Lemma \ref{L:complement1}). We can eliminate this point  
from $G(F)$ and the remaining mapping admits the same disparate selections. The relation of the 
disparate kernel of $F$ and elimination points of $F$ is described in subsequent statements.

\begin{lemma} \label{L:kernel.3}
Let $Q$ be a submapping of $F$, let $F^*$ be a submapping of $Q$ and let $W \subset X$ be a 
nonempty set. $G(F^*) \cap elim_Q(W) = \emptyset$.
\end{lemma}
\begin{proof} 
Let $(x,y) \in G(F^*) \cap ((X \backslash W) \times Y)$. Using Lemma \ref{L:kernel.1} ``(iii)", 
$(x,y) \in G(Q^*)$, i.e., there exists a disparate selection $s$ of $Q$ such that $s(x)=y$. By 
Lemma \ref{L:complement1}, $Q_{\{x\},\{(x,y)\}}$ admits a disparate selection and using 
Theorem \ref{T:main}, $(Q_{\{x\},\{(x,y)\}})_{\mid W}$ satisfies the generalized Hall condition. 
By definition of elimination points $(x,y) \notin elim_Q(W)$.
\end{proof} 

In particular, Lemma \ref{L:kernel.3} shows $elim_{F^*}(W) = \emptyset$ for each nonempty set 
$W \subset X$, i.e., $G(F^*)$ does not contain elimination points.

$F$ can be decomposed into the disparate kernel $F^\star$ of $F$ and elimination points.

\begin{theorem} \label{T:kernel.1}
$G(F^*) \cap elim_F(W) = \emptyset$ for each nonempty set $W \subset X$ and
\[
G(F)=G(F^*) \cup \bigcup_{W \subset X, W \ne \emptyset} elim_F(W).
\]
\end{theorem}
\begin{proof}
W.l.o.g. we may assume $\sharp X > 1$. Using Lemma \ref{L:kernel.3}, it is enough to show 
\[
G(F) \backslash G(F^*) \subset \bigcup_{W \subset X, W \ne \emptyset} elim_F(W).
\]
Let $(x,y) \in G(F) \backslash G(F^*)$. There does not exist a disparate selection $s$ of $F$ such that
$s(x)=y$. Using Lemma \ref{L:complement1}, $F_{\{x\},\{(x,y)\}}$ does not admit a disparate 
selection. Define $W=X \backslash \{x\} \ne \emptyset$ and using Theorem \ref{T:main},
$(x,y) \in elim_F(W)$.
\end{proof}

\begin{theorem} \label{T:kernel.2}
The following statements are equivalent:
\begin{enumerate}[(i)] 
\item $F$ is disparate. 
\item $G(F) \ne \emptyset$ and $elim_F(W) = \emptyset$ for each nonempty set $W \subset X$. 
\end{enumerate}
\end{theorem}
\begin{proof}
This statement follows from Theorem \ref{T:kernel.1}.
\end{proof}

The mapping $F^*$ can also be used for a unicity statement of disparate selections.

\begin{theorem} \label{T:kernel.3}
The following statements are equivalent:
\begin{enumerate}[(i)] 
\item $F$ admits a unique disparate selection.  
\item $F^*(x)$ is a singleton for each $x \in X$.  
\end{enumerate}
\end{theorem}
\begin{proof}
The equivalence of ``(i)" and ``(ii)" follows from the definition of $F^*$. 
\end{proof}
         
Theorem \ref{T:kernel.3} can be applied to Sudoku problems which are always uniquely solvable.

                                        %
                                        %
\section{Calculation of the disparate Kernel} \label{S:Calculation1}
We describe a method for the calculation of the disparate kernel $F^*$ of $F$. This method 
terminates after finitely many steps with a description of $F^*$ or it indicates that $F$ does not 
admit a disparate selection. In each cycle of the method the mapping $F$ is reduced by elimination 
points until there do not exist any more points to eliminate.

\begin{definition} \label{D:eliminationSet}
A nonempty set $W \subset X$ is called an elimination set of $F$ if $elim_F(W) \ne \emptyset$.
\end{definition}

Obviously, each elimination set is also a generalized critical set. Simple examples show the converse 
is not true. In contrast to the definition of generalized critical sets (Definition \ref{D:critical}) we require 
$(x,y)$ to be in $G(F_{\mid (X \backslash W)})$ and not only in $(X \backslash W) \times Y$. 
The search for elimination sets is the basic idea in the construction of Calculation Method 1.

Description of Calculation Method 1.
         
\begin{enumerate}[\qquad 1.]
\item 
Set $F^1=F$ and $i=1$.  
\medskip
\item
Choose an elimination set $W \subset X$ of $F^i$.  
\medskip
\item
If $F^i$ does not admit an elimination set. STOP.  \\
Otherwise continue.  
\medskip
\item 
Define $F^{i+1}$ by $G(F^{i+1})=G(F^i) \backslash elim_{F^i}(W)$.  
\medskip
\item 
If $F^{i+1}(x)=\emptyset$ for some $x \in X$. STOP. Otherwise continue.  
\medskip
\item 
Increase $i \rightarrow i+1$ and continue with Step 2.  
\medskip
\end{enumerate}          

Calculation Method 1 can be considered as conceptional, since it does not specify how to find an 
elimination set $W$ in Step 2. In \cite[Section 8]{Fis1} a method to determine an elimination set
had been proposed when the graph $(X,E)$ is complete.

The idea of using elimination sets in Calculation Method 1 is also used in strategies solving Sudoku 
puzzles (compare e.g. Crook \cite{Cro} or Provan \cite{Pro}).

\begin{lemma} \label{L:CM1.1}
Calculation Method 1 terminates after finitely many \\ steps.  
\end{lemma}
\begin{proof} 
Assume the calculation method does not stop at Step 3. Then $elim_{F^i}(W)$ is nonempty and is
contained in $G(F^i)$ for each $i$. Consequently, $\sharp G(F^{i+1})<\sharp G(F^i)$ and the 
method will stop at Step 5, since $X$ and $Y$ are finite sets.
\end{proof}

Based on Lemma \ref{L:CM1.1} we define $i^{max}_1$ to be the largest index which occurs 
in Calculation Method 1. If Calculation Method 1 stops at Step 3, $i \rightarrow i^{max}_1$.
If Calculation Method 1 stops at Step 5, $i+1 \rightarrow i^{max}_1$.

\begin{lemma} \label{L:CM1.2}
$G(F^*) \subset G(F^i)$ for each $i=1, \ldots, i^{max}_1$.
\end{lemma}
\begin{proof} 
We prove the statement by induction on $i=1, \ldots, i^{max}_1$. The statement is true for $i=1$. 
Let $2 \le i \le i^{max}_1$. By definition of Calculation Method 1 there exists an elimination set
$W \subset X$ of $F^{i-1}$. By induction hypothesis $G(F^*) \subset G(F^{i-1})$. Using Lemma 
\ref{L:kernel.3} with $Q=F^{i-1}$ yields $G(F^*) \cap elim_{F^{i-1}}(W) = \emptyset$ and we obtain
$G(F^*)=G(F^*) \backslash elim_{F^{i-1}}(W) \subset G(F^{i-1}) \backslash elim_{F^{i-1}}(W) = G(F^i)$.
\end{proof}

It is not necessary to know in advance if $F$ admits a disparate selection before starting with
Calculation Method 1. If $F$ admits a disparate selection Calculation Method 1 stops after finitely 
many steps with the disparate kernel $F^*$ of $F$. If $F$ does not admit a disparate selection,
Calculation Method 1 stops at Step 5 and indicates that $F$ does not admit a disparate selection.

\begin{theorem} \label{T:CM1.1}
Let $G(F) \ne \emptyset$. The following statements are equivalent: 
\begin{enumerate}[(i)]
\item $F$ admits a disparate selection.
\item Calculation Method 1 does not stop at Step 5.  
\item Calculation Method 1 stops at Step 3.
\item $F^{i^{max}_1}$ is disparate.
\item $G(F^{i^{max}_1}) \ne \emptyset$ and $F^*=F^{i^{max}_1}$.
\end{enumerate}
\end{theorem}
\begin{proof}
``(i) $\Rightarrow$ (ii)" $F^*$ has nonempty image sets and we apply Lemma \ref{L:CM1.2}.  \\
``(ii) $\Rightarrow$ (iii)" Using Lemma \ref{L:CM1.1}, Calculation Method 1 stops at Step 3.  \\
``(iii) $\Rightarrow$ (iv)" Using Theorem \ref{T:kernel.2}, $F^{i^{max}_1}$ is disparate since
$G(F^{i^{max}_1}) \ne \emptyset$ and $F^{i^{max}_1}$ does not admit an elimination set.  \\
``(iv) $\Rightarrow$ (v)" Using Lemma \ref{L:kernel.2}, $G(F^{i^{max}_1}) \subset G(F^*)$  
and using Lemma \ref{L:CM1.2} equality holds.  \\
``(v) $\Rightarrow$ (i)" $G(F^*) \ne \emptyset$, i.e., $F$ admits a disparate selection.
\end{proof}                           
             
If $F$ admits a unique disparate selection, $F^*(x)$ is a singleton for each $x \in X$ and the
result of Calculation Method 1 is a description of this unique disparate selection. If $F$ admits more 
than one disparate selection, we have to apply Calculation Method 2, which determines one of 
several disparate selections.

                                        %
                                        %
\section{Calculation of a Disparate Selection} \label{S:Calculation2}
We describe a method for the calculation of a disparate selection $s$ of $F$. This method makes use of 
Calculation Method 1 and terminates after finitely many steps with a disparate selection. In each cycle of 
the method the mapping $F$ is reduced to its disparate kernel and any point in the disparate kernel is 
selected. 

Description of Calculation Method 2.
         
\begin{enumerate}[\qquad 1.]
\item 
Set $F^1=F$ and $i=1$.
\medskip
\item
Determine the disparate kernel $(F^i)^*$ of $F^i$ with Calculation Method 1.
\medskip
\item
If Calculation Method 1 terminates at Step 5. STOP.  \\
Otherwise continue.
\medskip
\item
Choose $(x_i,y_i) \in G((F^i)^*)$.
\medskip
\item 
If $i=\sharp X$. STOP. Otherwise continue.
\medskip
\item 
Set $F^{i+1}=(F^i)^*_{\{x_i\},\{(x_i,y_i)\}}$, increase $i \rightarrow i+1$ and continue with Step 2.
\medskip
\end{enumerate}                                            

Please note, $F^{i+1}$ is defined on $X \backslash \{x_1, \ldots, x_i\}$ at each step of Calculation 
Method 2. Also note, by using the complement mapping in Step 6, $F^{i+1}$ may contain new 
elimination sets.

\begin{lemma} \label{L:CM2.1}
Calculation Method 2 terminates after finitely many steps.
\end{lemma}
\begin{proof} 
Assume the calculation method does not stop at Step 3. Then $i$ increases and the method will stop 
at Step 5, since $X$ is a finite set.
\end{proof}

Based on Lemma \ref{L:CM2.1} we define $i^{max}_2$ to be the largest index which occurs 
in Calculation Method 2.

\begin{lemma} \label{L:CM2.2}
Let $F$ admit a disparate selection. $F^i$ admits a disparate selection for $i=1, \ldots, i^{max}_2$.
\end{lemma}
\begin{proof} 
We prove the statement by induction on $i=1, \ldots, i^{max}_2$. The statement is true for $i=1$.
Let $2 \le i \le i^{max}_2$. By induction hypothesis $F^{i-1}$ admits a disparate selection. This
implies $(F^{i-1})^*$ is a disparate mapping and using Lemma \ref{L:complement1},
$F^i=(F^{i-1})^*_{\{x_{i-1}\},\{(x_{i-1},y_{i-1})\}}$ admits a disparate selection.
\end{proof}

\begin{theorem} \label{T:CM2.1}
The following statements are equivalent: 
\begin{enumerate}[(i)]
\item $F$ admits a disparate selection.
\item Calculation Method 2 does not stop at Step 3.  
\item Calculation Method 2 stops at Step 5.
\item ${i^{max}_2} = \sharp X$ and a disparate selection $s$ of $F$ is given by $s(x_i)=y_i$ for 
$i=1, \ldots, i^{max}_2$.
\end{enumerate}
\end{theorem}
\begin{proof}
``(i) $\Rightarrow$ (ii)" This implication follows from Lemma \ref{L:CM2.2} and Theorem \ref{T:CM1.1}
applied to $F^i$.  \\
``(ii) $\Rightarrow$ (iii)" Using Lemma \ref{L:CM2.1}, Calculation Method 2 stops at Step 5.  \\
``(iii) $\Rightarrow$ (iv)" Using the construction of Calculation Method 2, ${i^{max}_2} = \sharp X$, 
i.e. $X = \{x_1, \ldots, x_{i^{max}_2}\}$, $s$ is a selection of $F$ and by definition of the complement 
mapping, $s$ is disparate.  \\
``(iv) $\Rightarrow$ (i)" $X = \{x_1, \ldots, x_{i^{max}_2}\}$ and $F$ admits the disparate selection $s$.
\end{proof}

                                        %
                                        %
\section{Transitive Problems} \label{S:transitive}
We introduce an additional condition on our main model which allows us to simplify the generalized 
Hall condition and the definition of generalized critical sets.

It is our purpose to find a condition such that the existence of Hall collection characterizes the 
existence of a disparate selection for transitive problems. We start with a modified definition of 
critical sets and show the relation to generalized critical sets.

\begin{definition} \label{D:size}
The size of a set $G \subset X \times Y$ is defined by
\[
size(G) = max\{\sharp A \mid A \subset G \mbox{ is a disparate set}\}.
\]
\end{definition}

Using the definition of $size$, $F$ admits a Hall collection if and only if $size(G(F_{\mid W})) \ge \sharp W$ 
for each $W \subset X$.

\begin{definition} \label{D:t-critical}
A nonempty set $W \subset X$ is called a t-critical set of $F$ if there exists
$(x,y) \in (X \backslash W) \times Y$ such that 
\[
size(G((F_{\{x\},\{(x,y)\}})_{\mid W})) \le \sharp W - 1.
\]
\end{definition}

The next lemma shows that each t-critical set is a generalized critical set.

\begin{lemma} \label{L:trans.1}
Let $W \subset X$ be a t-critical set of $F$. $W$ is a generalized critical set of $F$.
\end{lemma}
\begin{proof} 
Let $(x,y) \in (X \backslash W) \times Y$ such that $size(G((F_{\{x\},\{(x,y)\}})_{\mid W})) \le \sharp W - 1$.
Then $(F_{\{x\},\{(x,y)\}})_{\mid W}$ does not admit a Hall collection.
\end{proof}

The converse of Lemma \ref{L:trans.1} is not true.

\begin{example} \label{E:trans.1}
We extend Example a) of Fig. \ref{Fig2} as depicted in Fig. \ref{Fig3}. We choose $(x,y)=(5,4)$ and 
observe that $W=\{1,2,3,4\}$ is a generalized critical set of $F$ since $(F_{\{5\},\{(5,4)\}})_{\mid W}$ 
does not admit a disparate selection. But $\{(1,2), (1,3), (3,1), (4,1)\}$ is a disparate set in 
$G((F_{\{5\},\{(5,4)\}})_{\mid W})$, i.e., $W$ is not a t-critical set of $F$.
\end{example}

\begin{figure}
\includegraphics{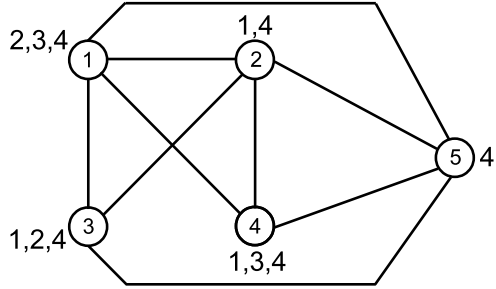}
\caption{Generalized critical but not t-critical set.} 
\label{Fig3}
\end{figure}

We extend the concept of minimal generalized critical sets to t-critical sets. A t-critical set $W \subset X$ 
of $F$ is called minimal if there does not exist a subset $W^\prime \subset W$, $W^\prime \ne W$, 
such that $W^\prime$ is a t-critical set of $F$.

\begin{lemma} \label{L:trans.2}
Let $W \subset X$ be a minimal generalized critical set and a t-critical set of $F$. $W$ is a 
minimal t-critical set of $F$.
\end{lemma}
\begin{proof} 
This is a consequence of Lemma \ref{L:trans.1}.
\end{proof}

We introduce a formal description of transitive problems.

\begin{definition} \label{D:transitive}
A graph $(X,E)$ is called transitive if $\{x,x^\prime \} \in E$, $\{x^\prime,x^{\prime\prime}\} \in E$ 
and $x \ne x^{\prime\prime}$ implies $\{x,x^{\prime\prime}\} \in E$.
\end{definition}

\begin{lemma} \label{L:trans.4}
Let $(X,E)$ be transitive, let $a, a^\prime, a^{\prime\prime} \in X \times Y$ such that $a, a^\prime$
are not disparate and $a^\prime, a^{\prime\prime}$ are not disparate.  $a, a^{\prime\prime}$ are not 
disparate. 
\end{lemma}
\begin{proof}
If $a=a^{\prime\prime}$, $a,a^{\prime\prime}$ are not disparate. Otherwise let $a=(x,y)$,
$a^\prime=(x^\prime,y^\prime)$ and $a^{\prime\prime}=(x^{\prime\prime},y^{\prime\prime})$. 
By assumption $y=y^\prime$, $\{x,x^\prime\} \in E$, $y^\prime = y^{\prime\prime}$ and 
$\{x^\prime,x^{\prime\prime}\} \in E$. This implies $y=y^{\prime\prime}$ and 
$\{x,x^{\prime\prime}\} \in E$, since $(X,E)$ is transitive and $a \ne a^{\prime\prime}$. Then 
$a=(x,y)$ and $a^{\prime\prime}=(x^{\prime\prime},y^{\prime\prime})$ are not disparate.
\end{proof}  

\begin{lemma} \label{L:trans.5}
Let $(X,E)$ be transitive, let $A \subset X \times Y$ be a disparate set and let $b \in X \times Y$.
$\sharp (A \cap hull(\{b\})) \le 1$. 
\end{lemma}
\begin{proof}
Let $a, a^\prime \in A \cap hull(\{b\})$. Using Lemma \ref{L:trans.4}, $a, a^\prime$ are not disparate.
Since $A$ is disparate, $a = a^\prime$.
\end{proof}

\begin{lemma} \label{L:trans.6}
Let $(X,E)$ be transitive, let $W \subset X$ be a nonempty set and let 
$(x,y) \in (X \backslash W) \times Y$.
\[
size(G(F_{\mid W})) \le size(G((F_{\{x\},\{(x,y)\}})_{\mid W})) + 1.
\]
\end{lemma}
\begin{proof}
Choose a set $A \subset G(F_{\mid W})$ such that $A$ is disparate and 
$\sharp A = size(G(F_{\mid W})) $. By definition of the complement mapping
\[
A \cap G((F_{\{x\},\{(x,y)\}})_{\mid W}) = A \cap compl(\{(x,y)\}).
\]
Using Lemma \ref{L:32} and \ref{L:trans.5} we obtain
\begin{align*}                                                   
size(G(F_{\mid W}))
& = \sharp A  \\
& = \sharp (A \cap compl(\{(x,y)\})) + \sharp (A \cap hull(\{(x,y)\}))  \\
& \le \sharp (A \cap G((F_{\{x\},\{(x,y)\}})_{\mid W}) + 1  \\
& \le size(G((F_{\{x\},\{(x,y)\}})_{\mid W})) + 1.
\end{align*}
\end{proof}        

\begin{lemma} \label{L:trans.7}
Let $(X,E)$ be transitive, let $G \subset X \times Y$ and let $A \subset hull(G)$ be a disparate set
such that $A \backslash G \ne \emptyset$. There exists $a \in A \backslash G$ and 
$g \in G \backslash A$ such that $A \backslash \{a\} \cup \{g\}$ is a disparate set.
\end{lemma}
\begin{proof}
By assumption there exists $a \in A \backslash G$ and this implies $a \in hull(G)$. By definition of 
the hull there exists $g \in G$ such that $a$ and $g$ are not disparate, i.e., $g \in G \backslash A$. 
The set $A \backslash \{a\}$ is disparate. Let $a^\prime \in A \backslash \{a\}$. Using Lemma 
\ref{L:trans.4}, $a^\prime$ and $g$ are disparate.
\end{proof} 

\begin{lemma} \label{L:trans.8}
Let $(X,E)$ be transitive and let $G \subset X \times Y$. $size(G) = size(hull(G))$.
\end{lemma}
\begin{proof}
Let $A \subset hull(G)$ be a disparate set such that $\sharp A = size(hull(G))$ and
$\sharp (A \cap G) \ge \sharp (A^\prime \cap G)$ for each disparate set $A^\prime \subset hull(G)$
such that $\sharp A^\prime = size(hull(G))$. Suppose $A \not\subset G$. Then 
$A \backslash G \ne \emptyset$ and we apply Lemma \ref{L:trans.7}. There exists $a \in A \backslash G$ 
and $g \in G \backslash A$ such that $A^\prime = A \backslash \{a\} \cup \{g\} \subset hull(G)$ is a 
disparate set. We obtain $\sharp A^\prime = \sharp A = size(hull(G))$ and 
$\sharp (A^\prime \cap G) = \sharp (A \backslash \{a\} \cap G) + 1 
= \sharp (A \cap G) +1  > \sharp (A \cap G)$ and this contradicts the assumption, i.e., $A \subset G$.
We conclude $size(hull(G)) = \sharp A \le size(G) \le size(hull(G))$.
\end{proof} 

\begin{lemma} \label{L:trans.9}
Let $(X,E)$ be transitive, let $W \subset X$ be a t-critical set of $F$ and let $\mathcal{H}$
be a collection of $F$. $\mathcal{H}$ is $W$-distributed.
\end{lemma}
\begin{proof}
By assumption there exists $(x,y) \in (X \backslash W) \times Y$ such that 
$size(G((F_{\{x\},\{(x,y)\}})_{\mid W})) \le \sharp W - 1$. Let 
$\mathcal{H} = \{A_V \subset G(F_{\mid V}) \mid V \subset X, V \ne \emptyset\}$ and let
$V \subset X \backslash W$ be a nonempty set. 
From set theory we obtain
\[
A_{W \cup V} \cap ((W \times Y) \cup (hull(\bigcup_{A \in \mathcal{H}_{\mid W}}A) \cap (V \times Y )))
\subset  A_{W \cup V} \cap hull(G(F_{\mid W})).
\]
Using Lemma \ref{L:trans.6} and \ref{L:trans.8} we obtain
\begin{align*}
& \sharp(A_{W \cup V} \cap ((W \times Y) \cup 
(hull(\bigcup_{A \in \mathcal{H}_{\mid W}}A) \cap (V \times Y ))))  \\
& \le \sharp(A_{W \cup V} \cap hull(G(F_{\mid W})))  \\
& \le size(hull(G(F_{\mid W})))  \\
& =  size(G(F_{\mid W}))  \\
& \le size(G((F_{\{x\},\{(x,y)\}})_{\mid W})) +1  \\
& \le \sharp W - 1 + 1  \\
& = \sharp W.
\end{align*}
\end{proof}  

\begin{lemma} \label{L:trans.11}
Let $(X,E)$ be transitive and let $\mathcal{H}$ be a Hall collection of $F$. $F$ satisfies the 
generalized Hall condition.
\end{lemma}
\begin{proof}
We prove the statement by induction on the number $\sharp X$ of elements in $X$. The statement is 
true for $\sharp X = 1$. 

Let $\sharp X > 1$, let $(W_1, \ldots, W_k)$, $k \in \mathbb{N}$, be a 
primitive $(F,\mathcal{H})$-critical cascade in $W_0=X$ and let $1 \le i \le k$. $W_i$ is a minimal 
generalized critical set of $F(W_1, \ldots, W_{i-1})$. Using Lemma \ref{L:cascade1} ``(ii)" there 
exists $(x,y) \in (pre(i) \backslash W_i) \times Y$ such that 
$(F(W_1, \ldots, W_{i-1})_{\{x\},\{(x,y)\}})_{\mid W_i}$
does not satisfy the generalized Hall condition. In particular, $\sharp W_i < \sharp X$ and by induction 
hypothesis $(F(W_1, \ldots, W_{i-1})_{\{x\},\{(x,y)\}})_{\mid W_i}$ does not admit a Hall collection. 
There exists a subset $W^\prime \subset W_i$ such that
\[
size(G((F(W_1, \ldots, W_{i-1})_{\{x\},\{(x,y)\}})_{\mid W^\prime})) \le \sharp W^\prime - 1,
\]
i.e., $W^\prime$ is a t-critical set of $F(W_1, \ldots, W_{i-1})$. Using Lemma \ref{L:trans.1},
$W^\prime$ is a generalized critical set of $F(W_1, \ldots, W_{i-1})$, i.e., $W_i = W^\prime$.

We apply Lemma \ref{L:cascade1} ``(i)", i.e., $\mathcal{H}(W_1, \ldots, W_{i-1})$ is a collection of  
$F(W_1, \ldots, W_{i-1})$ and $W_i$ is a t-critical set of $F(W_1, \ldots, W_{i-1})$. 
By Lemma \ref{L:trans.9}, $\mathcal{H}(W_1, \ldots, W_{i-1})$ is $W_i$-distributed and by 
Definition \ref{D:distributed2} and \ref{D:Hall2}, $F$ satisfies the generalized Hall condition.
\end{proof}  

When $(X,E)$ is transitive our main theorem from Section \ref{S:main} can be simplified.

\begin{theorem} \label{T:trans.1}
Let $(X,E)$ be transitive. The following statements are equivalent: 
\begin{enumerate}[(i)]
\item $F$ admits a Hall collection.
\item $F$ admits a disparate selection.
\end{enumerate}
\end{theorem}
\begin{proof}
The statement follows from Theorem \ref{T:main} and Lemma \ref{L:trans.11}
\end{proof} 

The characterization of Theorem \ref{T:trans.1} also had been obtained by Hilton and Johnson \cite{HJ2}
when every block of $(X,E)$ is a clique.

Lemma \ref{L:trans.11} also allows a new characterization of critical sets. We show that a minimal
generalized critical set is a minimal t-critical set for transitive problems, i.e., for minimal critical sets 
the converse of Lemma \ref{L:trans.1} becomes true.

\begin{lemma} \label{L:trans.12}
Let $(X,E)$ be transitive and let $W \subset X$ be a generalized critical set of $F$. There exists a 
t-critical set  $W^\prime \subset W$ of $F$.
\end{lemma}
\begin{proof}
By assumption there exists $(x,y) \in (X \backslash W) \times Y$ such that 
$(F_{\{x\},\{(x,y)\}})_{\mid W}$ does not satisfy the generalized Hall condition. Using 
Lemma \ref{L:trans.11}, $(F_{\{x\},\{(x,y)\}})_{\mid W}$ does not admit a Hall collection, i.e., there exists a subset 
$W^\prime \subset W$ such that 
$size(G((F_{\{x\},\{(x,y)\}})_{\mid W^\prime})) \le \sharp W^\prime - 1$.
$W^\prime$ is a t-critical set of $F$.
\end{proof} 

\begin{lemma} \label{L:trans.13}
Let $(X,E)$ be transitive and let $W \subset X$. The following statements are equivalent:
\begin{enumerate}[(i)]
\item $W$ is a minimal generalized critical set of $F$.
\item $W$ is a minimal t-critical set of $F$.
\end{enumerate}
\end{lemma}
\begin{proof}
``(i) $\Rightarrow$ (ii)" Using Lemma \ref{L:trans.12}, there exists a t-critical set 
$W^\prime \subset W$ of $F$. By Lemma \ref{L:trans.1}, $W^\prime$ is a generalized critical set of $F$,
i.e., $W^\prime = W$ and $W$ is a t-critical set of $F$. The statement follows from Lemma 
\ref{L:trans.2}.  \\
``(ii) $\Rightarrow$ (i)"  Using Lemma \ref{L:trans.1}, $W$ is a generalized critical set of $F$. Let
$W^\prime \subset W$ be a generalized critical set of $F$. Using Lemma \ref{L:trans.12} there exists a 
subset $W^{\prime\prime} \subset W^\prime$ such that $W^{\prime\prime}$ is a t-critical set of $F$.
This implies $W^{\prime\prime} = W^\prime = W$, i.e., $W$ is a minimal generalized critical set of $F$.
\end{proof} 

Using Lemma \ref{L:trans.6} we derive a necessary condition for t-critical sets.

\begin{lemma} \label{L:trans.14}
Let $(X,E)$ be transitive and let $W \subset X$ be a t-critical set of $F$.
$size(G(F_{\mid W})) \le \sharp W$.
\end{lemma}
\begin{proof}
Apply Lemma \ref{L:trans.6}.
\end{proof} 

The condition of Lemma \ref{L:trans.14} is not sufficient for $W$ to be a t-critical set of $F$, 
in general.

If $W \subset X$ induces a complete subgraph of $(X,E)$ the size function can be replaced by the 
cardinality.

\begin{lemma} \label{L:trans.subgraph.1}
Let $(X,E)$ be transitive, let $W \subset X$ be a nonempty set such that the subgraph of $(X,E)$ induced 
by $W$ is complete and let $A \subset W \times Y$ be a disparate set.
\[
\sharp A = \sharp \{y \in Y \mid \mbox{there exists } x \in W \mbox{ such that } (x,y) \in A\}.
\]
\end{lemma}
\begin{proof}
The statement is true if $A = \emptyset$ and it is clear that
\[
\sharp A \ge \sharp \{y \in Y \mid \mbox{there exists } x \in W \mbox{ such that } (x,y) \in A\}.
\]
Let $A \ne \emptyset$ and set $k = \sharp A \ge 1$. There exist $(x_i,y_i) \in W \times Y$, $i=1, \dots , k$ 
such that $A = \{(x_1,y_1), \ldots (x_k,y_k)\}$. Since $A$ is disparate and the subgraph of $(X,E)$ 
induced by $W$ is complete, $y_i \ne y_j$ for $1 \le i,j \le k$, $i \ne j$, i.e., 
$\sharp \{y_1, \ldots , y_k\} = k$ and the statement is proved.
\end{proof} 

\begin{lemma} \label{L:trans.subgraph.2}
Let $(X,E)$ be transitive and let $W \subset X$ be a nonempty set such that the subgraph of $(X,E)$ 
induced by $W$ is complete.
\begin{enumerate}[(i)]
\item $\sharp F(W) = size(G(F_{\mid W}))$.
\item $\sharp (F(W) \backslash \{ y \}) = size(G((F_{\{x\},\{(x,y)\}})_{\mid W}))$  \\
for each $(x,y) \in (X \backslash W) \times Y$.
\end{enumerate}
\end{lemma}
\begin{proof}
This is  consequence of Lemma \ref{L:trans.subgraph.1}.
\end{proof}

                                        %
                                        %
\section{A Theorem of Ryser} \label{S:ryser}
We consider a Latin square of size $n$ with values in $\mathbf{N} = \{1, \ldots, n\}$ where in the upper 
left corner a rectangle of $r$ rows and $s$ columns is prepopulated in such a way that the values in
each row and each column are distinct. Ryser \cite{Rys} posed and answered the question if it is possible 
to extend this partial Latin square to a partial Latin square of size $r \times n$.

We derive a new proof of this result based on Theorem \ref{T:trans.1}. M. Hall \cite{HalM1} already 
proved that a partial Latin square of size $r \times n$ can always be completed to a Latin square 
of size $n \times n$.

We denote by $N(v)$ the number of occurences of a value $v \in \mathbf{N}$ in the rectangle and 
$R(v)$ denotes the set of all rows in the rectangle where the value $v$ is not contained in the rectangle. 
Then $N(v) + \sharp R(v) = r$ for each $v \in \mathbf{N}$. We define
\[  
X = \bigcup_{\substack{v \in \mathbf{N}, \\ R(v) \ne \emptyset}} (\{v\} \times R(v)),
Y = \{s+1, \ldots , n\},
F(x) = Y \mbox{ for } x \in X
\]
and $E = \{\{(v,r_1), (v,r_2)\} \subset X \mid v \in \mathbf{N}, r_1, r_2 \in R(v), r_1 \ne r_2 \}$.

This model is of the type introduced in Section \ref{S:disparate} and we show it satisfies the property
of Section \ref{S:transitive}.

\begin{lemma} \label{L:ryser.1}
$(X,E)$ is transitive.
\end{lemma}
\begin{proof}
Let $x, x^\prime, x^{\prime\prime} \in X$, $\{x,x^\prime \} \in E$, 
$\{x^\prime,x^{\prime\prime}\} \in E$ and $x \ne x^{\prime\prime}$. By definition of X
there exist $v, v^\prime, v^{\prime\prime} \in \mathbf{N}$, $r \in R(v)$, $r^\prime \in R(v^\prime)$
and $r^{\prime\prime} \in R(v^{\prime\prime})$ such that $x=(v,r)$, 
$x^\prime=(v^\prime,r^\prime)$, $x^{\prime\prime}=(v^{\prime\prime},r^{\prime\prime})$.
By definition of $E$, $v=v^\prime$ and $v^\prime = v^{\prime\prime}$. This shows
$x^{\prime\prime}=(v,r^{\prime\prime})$, $r^{\prime\prime} \in R(v)$ and
$r \ne r^{\prime\prime}$, i.e., $\{x,x^{\prime\prime} \} = \{(v,r), (v,r^{\prime\prime})\} \in E$.
\end{proof} 

We identify the extension of the partial Latin square with a disparate selection of the set-valued mapping 
$F$. By definition of our model the $r \times s$-rectangle can be extended if and only if $F$ admits a
selection $s$ such that $s(v,r_1) \ne s(v,r_2)$ for each $v \in \mathbf{N}$ and $r_1, r_2 \in R(v)$,
$r_1 \ne r_2$. By definition of $E$ this is true if and only if $s$ is disparate. The value $v$ is placed 
in row $r$ and column $s(v,r)$ for each $(v,r) \in X$.

\begin{lemma} \label{L:ryser.2}
The following statements are equivalent: 
\begin{enumerate}[(i)]
\item $F$ admits a Hall collection.
\item $N(v) \ge r+s-n$ for each $v \in \mathbf{N}$. 
\end{enumerate}
\end{lemma} 
\begin{proof}
``(i) $\Rightarrow$ (ii)" Let $v \in \mathbf{N}$. We distinguish two cases. \\
Case 1: $R(v) = \emptyset$. \\
Then  $r - N(v) = \sharp R(v) = 0 \le n-s$. \\
Case 2: $R(v) \neq \emptyset$. \\
Define $W = \{v\} \times R(v)$. By definition of $E$, $W$ induces a complete subgraph of $(X,E)$.
Using ``(i)" and Lemma \ref{L:trans.subgraph.2} ``(i)" we obtain 
$r - N(v) 
 = \sharp R (v) 
 = \sharp W 
 \le size(G(F_{\mid W})) 
 = \sharp F(W) 
 \le \sharp Y 
 = n-s$.  \\ 
``(ii) $\Rightarrow$ (i)"  Let $W \subset X$ be a nonempty set. There exists a minimal subset 
$J \subset \mathbf{N}$ such that $W \subset \bigcup_{v \in J} (\{v\} \times R(v))$. Choose a row 
$r_v \in R(v)$ such that $(v,r_v) \in W$ for each $v \in J$ and define 
$A = ( \bigcup_{v \in J} \{(v,r_v)\}) \times Y$.
Then $A \subset G(F_{\mid W})$ is a disparate set and we obtain
\[
\sharp A
= \sharp Y \cdot \sharp J 
=  (n - s) \cdot \sharp J 
\ge \sum_{v \in J} (r - N(v)) 
\]
\[
= \sum_{v \in J} \sharp R (v) 
= \sum_{v \in J} \sharp (\{v\} \times R(v)) 
\ge \sharp W.
\]
\end{proof}

The implication ``(i) $\Rightarrow$ (ii)" of Lemma \ref{L:ryser.2} also had been proved by Bobga 
and Johnson \cite{Bob3} (crediting the proof to Hilton and Johnson \cite{HJ1}).  
The implication ``(ii) $\Rightarrow$ (i)" can be derived from Ryser's result.

With these preliminaries we are able to derive the result of Ryser \cite{Rys} from our general result.

\begin{theorem}[Ryser \cite{Rys}] \label{T:ryser.1}
We consider a Latin square of size $n$ with values in $\mathbf{N}$ where in the upper left corner 
a rectangle of $r$ rows and $s$ columns is prepopulated in such a way that the values in each 
row and each column are distinct. The following statements are equivalent: 
\begin{enumerate}[(i)]
\item The Latin rectangle can be extended to a Latin rectangle of size $r \times n$.
\item $N(v) \ge r+s-n$ for each $v \in \mathbf{N}$. 
\end{enumerate}
\end{theorem}
\begin{proof}
Using Lemma \ref{L:ryser.1} and \ref{L:ryser.2} we can apply Theorem \ref{T:trans.1}.
\end{proof}

                                        %
                                        %
\section{Theorem of Hall} \label{S:Hall3}
In this section we describe the classical marriage theorem of P. Hall \cite{HalP} in our terminology.
This theorem is part of our model when $(X,E)$ is complete. If $(X,E)$ is complete, $(X,E)$ is transitive
and each disparate selection describes a complete system of distinct representatives.

\begin{theorem}[P. Hall \cite{HalP}]  \label{T:hall.1}
Let $(X,E)$ be complete. The following statements are equivalent: 
\begin{enumerate}[(i)]
\item $\sharp F(W) \ge \sharp W$ for each $W \subset X$.
\item $F$ admits a disparate selection.
\end{enumerate}
\end{theorem}
\begin{proof}
The statement follows from Theorem \ref{T:trans.1} when $(X,E)$ is complete and Lemma 
\ref{L:trans.subgraph.2} ``(i)".
\end{proof} 

If $(X,E)$ is complete, t-critical sets of $F$ can be defined in a manner they had been used 
by Halmos and Vaughan \cite{HV} in their proof of the marriage theorem.

\begin{lemma} \label{L:hall.2}
Let $(X,E)$ be complete and $W \subset X$. The following statements are equivalent: 
\begin{enumerate}[(i)]
\item $W$ is a t-critical set of $F$.
\item $W \ne \emptyset$, $W \ne X$ and $\sharp F(W) \le \sharp W$.
\end{enumerate}
\end{lemma} 
\begin{proof}
Using Lemma \ref{L:trans.subgraph.2} ``(ii)", $W$ is a t-critical set of $F$ if and only if $W \ne \emptyset$ and 
there exists $(x,y) \in (X \backslash W) \times Y$ such that 
$\sharp (F(W) \backslash \{ y \}) \le \sharp W - 1$. This is true if and only if $W \ne \emptyset$,
$W \ne X$ and $\sharp F(W) \le \sharp W$.
\end{proof}

                                        %
                                        %

                                        %
                                        %

\end{document}